\documentclass[a4paper]{article}
\usepackage{hyperref}
\usepackage[english]{babel}
\usepackage{indentfirst}
\usepackage{fancyhdr}
\usepackage{graphicx,subcaption}
\usepackage{amssymb}
\usepackage{amsmath}
\usepackage{cases}
\usepackage{mathtools}
\usepackage{amsthm}
\usepackage{color}
\usepackage{epstopdf}
\usepackage{thm-restate}

\newtheorem{theorem}{Theorem}

\newtheorem{lemma}[theorem]{Lemma}
\newtheorem{remark}{Remark}

\usepackage{tikz}

\usetikzlibrary{positioning}
\usetikzlibrary{arrows}
\usetikzlibrary{arrows.meta}
\usetikzlibrary{calc}
\usetikzlibrary{shapes}
\usepackage{verbatim}

\def\R0{\mathcal{R}_0}

\newcommand{\correz}[1]{\textcolor{black}{#1}}

\title{A geometric analysis of the SIRS compartmental model with fast information and misinformation spreading}
\author{Iulia Martina Bulai$^{1}$, Mattia Sensi$^{2}$, Sara Sottile$^{3}$\\[1em]
$^1${\footnotesize Dipartimento di Scienze Chimiche, Fisiche, Matematiche e Naturali, Universit\`a degli Studi di Sassari,}\\ {\footnotesize Via Vienna 2, 07100 Sassari, Italy}\\
$^2${\footnotesize Department of Mathematical Sciences ``G. L. Lagrange'', Politecnico di Torino,}\\ {\footnotesize Corso Duca degli Abruzzi 24, 10129 Torino Italy}\\
$^3${\footnotesize Dipartimento di Matematica, Universit\`a degli Studi di Trento,} \\{\footnotesize Via Sommarive 14, 38123 Povo (Trento), Italy}}
\date{\today}

\begin{document}

\maketitle

\begin{abstract}
\correz{We propose a novel slow-fast SIRS compartmental model with demography, by coupling a slow disease spreading model and a fast information and misinformation spreading model. Beside the classes of susceptible, infected and recovered individuals of a common SIRS model, here we define three new classes related to the information spreading model, e.g. unaware individuals, misinformed individuals and individuals who are skeptical to disease-related misinformation. Under our assumptions, the system evolves on two time scales. We completely characterize its asymptotic behaviour with techniques of Geometric Singular Perturbation Theory (GSPT). We exploit the time scale separation to analyse two lower dimensional subsystem separately. First, we focus on the analysis of the fast dynamics and we find three equilibrium point which are feasible and stable under specific conditions. We perform a theoretical bifurcation analysis of the fast system to understand the relations between these three equilibria when varying specific parameters of the fast system. Secondly, we focus on the evolution of the slow variables and we identify three branches of the critical manifold, which are described by the three equilibria of the fast system. We fully characterize the slow dynamics on each branch. Moreover, we show how the inclusion of (mis)information spread may negatively or positively affect the evolution of the epidemic, depending on whether the slow dynamics evolves on the second branch of the critical manifold, related to the skeptical-free equilibrium or on the third one, related to misinformed-free equilibrium, respectively. We conclude with numerical simulations which showcase our analytical results.}


\end{abstract}


\section{Introduction}

In the era of social networks, when information travels fast between continents, it is of paramount importance to understand how the evolution of a disease can be affected by human behavioral dynamics influenced by information diffusion, \cite{DurMulSal,PhysRevE.102.042314,sontag2022misinformation}. 
For decades, from the early 20th century, the evolution of epidemics are modelled and studied via ordinary differential equations (ODEs) systems. The compartmental models are important tools for a better understanding of infectious diseases and they have been introduce in 1927 by Kermack and McKendrick \cite{KerMcK}, in fact they can be used to predict how the disease spread, or obtain information on the duration of an epidemic, the number of infected individuals, etc., but also to identify optimal strategies for control the disease.

In particular, mathematicians working on epidemic modelling have recently started focusing on the inclusion of awareness (with respect to the infectious disease under study) in the population \cite{juher2018tuning,juher2023saddle,just2018oscillations,manrubia2022individual,ye2021game}. Individuals aware of the disease might be more cautious towards risky behaviours or, \correz{on the other hand,} be skeptical of mainstream information and ignore precautions. The inclusion of distinct information based behaviours allows for more complex dynamics to be described, analysed and simulated, \cite{BulMonPed,MonProBul} and it is especially relevant in recent years, in which public opinion on COVID-19 and the corresponding containment measures reflected on the evolution of the pandemic, \cite{JohPel,ScaJacKat,GalValCas}.

In \cite{WANG20151}, a review concerning nonlinear coupling between disease dynamics and human behavioral dynamics, one of the results shows that emergent self-protective behavior can dampen an epidemic. \correz{What can one say} about the surviving of a virus in an epidemic due to human behavioral dynamics related to misinformation about the real evolution of the disease? \correz{To the best of our knowledge, in the literature (e.g. \cite{BulMonPed,MonProBul,buonomo2023oscillations}), the models that consider both human behavioral and disease dynamics, respectively, have maintained the simple structure of an ODEs system. They are models with all the variables characterized by one time scale, without considering the fact that the disease might spread slower than opinions and different time scales could render the model more realistic.}

In this paper, we focus on the interplay between fast information spreading and slow(er) disease spreading using techniques from Geometric Singular Perturbation Theory (GSPT).
Since the pioneering papers written by N. Fenichel~\cite{fenichel1979geometric}, GSPT has proven extremely suitable to describe systems evolving on multiple time scales, and analyse their transient and asymptotic behaviours. A thorough description of the fundamentals of GSPT, and hence of the techniques we use in this paper, can be found in \cite{hek2010geometric,jones1995geometric,kuehn2015multiple}; furthermore, for a more concise introduction, we refer the interested reader to the introductory sections of \cite{jardon2021geometric1}. 

The presence of widely different time scales is ubiquitous in natural phenomena \cite{bertram2017multi,hek2010geometric,jones1995geometric,kuehn2015multiple,wechselberger2020geometric}, and in particular in epidemic spreading. The latter has been the topic of various recent papers concerning the various facets of mathematical modelling of infectious diseases, particularly through the use of GSPT \cite{brauer2019singular,heesterbeek1993saturating,jardon2021geometric1,jardon2021geometric2,rocha2016understanding,wang2014dynamical}.

In such a setting, separation of time scales might occur in various ways.
Consider, for example, models which include both disease and demographic dynamics: typically, infectious periods have a much shorter duration than the average lifespan of the individuals in the population (weeks vs. years) \cite{andreasen1993effect,jardon2021geometric1,jardon2021geometric2}. A similar separation may occur between the infectiousness window of a communicable disease and the corresponding duration of natural or vaccine induced immunity \cite{chaves2007loss,woolthuis2017variation}.
Individual behaviour, including but not limited to (mis)information spreading and changes in the strategy towards an ongoing epidemic, may also evolve much faster than epidemics; several papers focus on this, both in continuous \cite{buonomo2023oscillations,castillo2016perspectives,dellamarca2023geometric,schecter2021geometric} and discrete time \cite{de2020discrete,bravo2021discrete}.

The main goal of this paper is to introduce and study \correz{a novel} SIRS compartmental model with demography and fast information and misinformation spreading in the population. Considering the speed at which information spreads in the age of social media, we let our system evolve on two time scales, a fast one, corresponding to the information ``layer'' and a slow one, corresponding to the epidemic ``layer''. We completely characterize the possible asymptotic behaviours of the system we propose with techniques of GSPT. In particular, we emphasise how the inclusion of (mis)information spreading can radically alter the asymptotic behaviour of the epidemic, depending on whether a non-negligible part of the population is misinformed or skeptical to misinformation. Indeed, from the numerical simulations we observe that the (mis)information spread can positively or negatively affect the evolution of the epidemic. In the first case, the output of the model can change from an endemic scenario to the extinction of the disease one, while in the second case, more interestingly, the misinformation spread can transform an infectious disease expected to die out into an endemic one.

The paper is structured as follows: in Section \ref{sec:model}, we propose the coupled information-infection spreading model as a singularly perturbed system of ODEs in standard form. In Section \ref{sec:fast}, we completely characterize the fast dynamics of our model, providing a comprehensive list of equilibria and the corresponding local and global stability conditions. In Section \ref{sec:slow}, we analyse the behaviour of the system on each branch of the critical manifold (i.e. equilibria of the fast system). In Section \ref{sec:entry-exit}, we briefly recall the definition of the entry-exit function, and we exploit it to deduce information on the possible delayed loss of stability in our system. In Section \ref{sec:numerics}, we perform an extensive numerical exploration of our model, illustrating our analytical results with carefully selected examples. We conclude in Section \ref{sec:concl}.

\section{The coupled info-epidemic model}\label{sec:model}

In this section, we introduce \correz{a novel} mathematical model by coupling two systems of ordinary differential equations, infodemic and epidemic systems, respectively. We assume that the aforementioned systems evolve on two widely different time scales, a fast one and a slow one, respectively. This assumption will shortly lead to a coupled fast-slow system. The Ordinary Differential Equations (ODEs) model describes the evolution of the considered compartments in time. We will denote the population in each compartment with a capital letter.

We first introduce a UMZ (Unaware-Misinformed-Skeptical) mathematical model, which describes the spread of (mis)information about an infectious disease through a compartmental SI$_1$I$_2$-like (Susceptible -- Infected with disease 1, representing misinformation -- Infected with disease 2, representing skepticism towards fake news) model.

The three classes of the infodemic model are 
\begin{itemize}
\item $U$: unaware individuals about the infectious disease related topics;
\item $M$: misinformed individuals, inclined to underestimate the efficiency of the measures put in place to reduce the number of contagions such as non-pharmaceutical measures (NPI), vaccines, etc., or even refuse these measures entirely;
\item $Z$: individuals skeptical to the disease related misinformation and inclined to take the proper measures in order to avoid the infection.
\end{itemize}
We assume that the unaware individuals, $U$, once they ``receive'' the misinformation can either agree and become misinformed, or become skeptical, therefore moving to the $M$ class at a rate $b_1$ or to the $Z$ class at a variable rate $\bar b_2$, which depends on the infected individuals $I$, introduced below, respectively. In this model, we also consider the  parameter \correz{$\mu_1$}, that is the loss of opinion (or interest in the ongoing epidemic) rate of the population. \correz{The total population is $N(t)=U(t)+M(t)+Z(t)$ for all $t \ge 0$}. 

The infodemic model (fast system \correz{in the fast time variable $t$}) reads 
\begin{subequations}\label{mod_1}
\begin{eqnarray}
\frac{\text{d}U}{\text{d} t}&=&  \correz{\mu_1 N}   -b_1\correz{\dfrac{UM}{N}}- \correz{\bar{b_2}(I) \dfrac{UZ}{N}} \correz{-\mu_1  U},\label{mod_1_U}\\ 
\frac{\text{d}M}{\text{d} t}&=& \correz{ b_1 \correz{\dfrac{UM}{N}} -\mu_1 M},\label{mod_1_M}   \\
\frac{\text{d}Z}{\text{d}t}&=& \correz{ \bar{b_2}(I)\correz{\dfrac{UZ}{N}} -\mu_1  Z} , \label{mod_1_Z}
\end{eqnarray}
\end{subequations}
with \correz{the first coupling term between the two subsystems} 
$$\bar{b_2}\correz{(I)} \coloneqq \dfrac{b_2}{1-KI}$$ and $K\in  [0,1)$. The positive feedback function $\bar{b_2}$ is an increasing function of $I$ with maximal rate $b_2$. This feedback function \correz{reproduces} the average perception of the disease severity, similarly to the strategy adopted in \cite{BulMonPed,MonProBul}, where the authors have used a Michaelis-Menten expression to mimic the societal response to the recent COVID-19 pandemic. \correz{Notice that the total population $N$ in system \eqref{mod_1} remains constant, as can be seen
by observing $N'(t) = 0$, and we can assume, without loss of generality, that $N(0) = 1$ so that $N(t) = 1$ for all $t\ge 0$.} 

Secondly, we introduce a modified version of the classical SIRS (Susceptible -- Infected -- Recovered -- Susceptible) model to describe the infection dynamics, and taking into account that the infection rate can change in time, e.g. increase when the number of misinformed individuals $M$ increase and decrease when the number of skeptical individuals $Z$ increase, respectively.
Here we denote by $S$ the class of susceptible individuals, with $I$ the class of infected individuals and with $R$ of recovered, respectively. We denote with $\varepsilon \ll 1$ a small parameter that \correz{describes the ratio of the two time scales, meaning that the slow dynamics (epidemic spreading, recovery etc.) happens at a much slower rate than the information spreading (fast system \eqref{mod_1})}. 
By introducing the slow time scale $\tau = \varepsilon t$, the epidemic model (slow system \correz{in the slow time variable $\tau$}) reads:
\begin{subequations}\label{mod_2}
\begin{eqnarray}
\frac{\text{d}S}{\text{d} \tau}&=&  \correz{\mu_2 N
   -\bar{\beta}(M,Z) \dfrac{S I}{N}}+\eta R- \correz{\mu_2} S,\label{mod_2_S}\\ 
\frac{\text{d}I}{\text{d} \tau}&=&  \correz{\bar{\beta}(M,Z) \dfrac{S I}{N}}- \gamma I-\correz{\mu_2}  I,\label{mod_2_I}   \\
\frac{\text{d}R}{\text{d} \tau}&=& \gamma I- \eta R-\correz{\mu_2} R \label{mod_2_R}, 
\end{eqnarray}
\end{subequations}
with \correz{the second coupling term between the two subsystems}
$$
    \correz{\bar{\beta}(M,Z)}\coloneqq \beta \dfrac{1+M}{1+Z}.
$$
\correz{The total population in system \eqref{mod_2} is $N(\tau) = S(\tau) + I(\tau) + R(\tau)$ and $N'(\tau) = 0$ for all $\tau \geq 0$. Assuming that the total population at time $\tau=0$ satisfies $N(0)=S(0)+I(0)+R(0)=1$, then $N(\tau)\equiv 1$ for all $\tau \geq 0$.}

All the parameters we describe in the following are given with respect to the slow time scale $\tau$; hence, all these parameters are $\mathcal{O}(\varepsilon)$ small compared to the ones we introduced for system \eqref{mod_1}. The first equation of \eqref{mod_2} describes the evolution in time of susceptible individuals. We assume to have a birth/immigration term \correz{$\mu_2$}; susceptible individuals can get infected at a variable rate $\correz{\Bar{\beta}(M,Z)}$, once recovered (third equation) they can lose their immunity at a rate $\eta$. In each compartment, they can die at a constant rate \correz{$\mu_2$}; this means that, for simplicity, we are ignoring disease induced mortality. 
The second equation of \eqref{mod_2} describes the evolution of the infected individuals, which recover at a rate $\gamma$.

\correz{Recall that $N(t)\equiv 1$ for all $t\geq 0$. Then,} the coupled model \eqref{mod_1}-\eqref{mod_2} in the fast time scale can be written as
\begin{subequations}\label{mod_3}
\begin{eqnarray}
 \frac{\text{d}U}{\text{d}t}&=&  \correz{\mu_1}  -b_1UM- \bar{b_2}\correz{(I)}UZ \correz{ -\correz{\mu_1} U}, \label{mod_3_U}\\
 \frac{\text{d}M}{\text{d}t}&=& \correz{ b_1 U M -\correz{\mu_1} M }, \label{mod_3_M}  \\
 \frac{\text{d}Z}{\text{d}t}&=& \correz{ \bar{b_2}(I) U Z -\correz{\mu_1} Z} , \label{mod_3_Z}\\
 \frac{\text{d}S}{\text{d}t}&=& \varepsilon \left(  \correz{\mu_2} - \correz{\bar{\beta}(M,Z)} S I+\eta R- \correz{\mu_2}S \right), \label{mod_3_S}\\
\frac{\text{d}I}{\text{d}t}&=& \varepsilon \left(  \correz{\bar{\beta}(M,Z)} S I- \gamma I-\correz{\mu_2} I \right),   \label{mod_3_I}\\
\frac{\text{d}R}{\text{d}t}&=& \varepsilon \left( \gamma I- \eta R-\correz{\mu_2} R \right) 
\label{mod_3_R} \end{eqnarray}
\end{subequations}
\correz{If $\varepsilon=0$ in system \eqref{mod_3}, populations $S$, $I$ and $R$ remain constant, and we recover the fast system \eqref{mod_1}. Equilibria of system \eqref{mod_1} form the so-called \emph{critical manifold} in the full 6 dimensional space.} Due to their biological interpretation, we assume that all the parameters are non-negative.

Moreover, it is easy to check that
$$
\frac{\text{d}X}{\text{d}t}|_{X=0}\geq 0,
$$
for $X=U,M,Z,S,I,R$, so that no population can ever become negative.

\begin{remark}
Notice that for $K = 0$, $\bar{b_2} \equiv b_2$, meaning that the evolution of the info-epidemic model does not depend on the number of infected individuals. However, the epidemic model depends on the number of misinformed and skeptical individuals, respectively, as they both have an impact on the infection rate. In this particular case, the misinformed individuals can change their mind depending only on the news received on the social media or other communication channels, or on their believes, and not based on the real-time evolution of the epidemic.
    \label{remark_K}
\end{remark}

In Figure \ref{fig:flow}, a flow diagram for the coupled system \eqref{mod_3} illustrates the mutual relationships among the two layers, i.e. (mis)information and disease spreading, considered in the population. The solid lines represent movement of individuals between compartments, while the dashed lines represent influences between information and epidemic spreading. 
Outgoing arrows denote losses due to deaths.

\correz{System \eqref{mod_3} could be reduced from 6 to 4 ODEs by exploiting twice the constant of motion $N$. However, doing so does not significantly simplify the analytical analysis. Moreover, not considering two of the variables of the model impairs a clear exposition of the results. Moreover, we are able to fully characterize the asymptotic behaviour of both systems \eqref{mod_1} and \eqref{mod_2}; hence, we remark that such a dimensionality reduction is viable but not necessary in this particular case.}

\begin{figure}[htbp!]
			\centering
		\begin{tikzpicture}
		\node[draw,circle,red,thick,minimum size=.75cm] (u) at (1,0) {$U$};
		\node[draw,circle,red,thick,minimum size=.75cm] (m) at (3,0) {$M$};
		\node[draw,circle,red,thick,minimum size=.75cm] (z) at (3,-2) {$Z$};
		\node[draw,circle,blue,thick,minimum size=.75cm] (s) at (5,0) {$S$};
		\node[draw,circle,blue,thick,minimum size=.75cm] (i) at (5,-2) {$I$};
        \node[draw,circle,blue,thick,minimum size=.75cm] (r) at (7,-2) {$R$};
		\draw[-{Latex[length=2.mm, width=1.5mm]},thick] (u)--(m) node[above, midway]{$b_1 UM$};

        \draw[-{Latex[length=2.mm, width=1.5mm]},thick] (u)--(z) node[below, midway]{$\bar{b_2} UZ\qquad$};
		\draw[-{Latex[length=2.mm, width=1.5mm]},thick] (i)--(r) node[above, midway]{$\gamma I$};

  \draw[-{Latex[length=2.mm, width=1.5mm]},thick] (s)--(i) node[right, midway]{$\bar{\beta}SI$};

    \draw[-{Latex[length=2.mm, width=1.5mm]},thick] (r)--(s) node[right, midway]{$\;\eta R$};
		\draw[{Latex[length=2.mm, width=1.5mm]}-,thick] (s)--++(0,1.5) node[left,midway]{$\correz{\mu_2}$};
		\draw[-{Latex[length=2.mm, width=1.5mm]},thick] (s)--++(1.5,0) node[above,midway]{$\correz{\mu_2} S$};
		\draw[-{Latex[length=2.mm, width=1.5mm]},thick] (i)--++(0,-1.5) node[right,midway]{$\correz{\mu_2} I$};
		\draw[-{Latex[length=2.mm, width=1.5mm]},thick] (r)--++(1.5,0) node[above,midway]{$\correz{\mu_2} R$};
		\draw[{Latex[length=2.mm, width=1.5mm]}-,thick] (u)--++(-1.5,0) node[above,midway]{$\correz{\mu_1}$};
        \draw[-{Latex[length=2.mm, width=1.5mm]},thick] (m)--++(0,1.5) node[left,midway]{$\correz{\mu_1} M$};
        \draw[-{Latex[length=2.mm, width=1.5mm]},thick] (u)--++(0,1.5) node[left,midway]{$\correz{\mu_1}U$};

    \draw[-{Latex[length=2.mm, width=1.5mm]},thick] (z)--++(0,-1.5) node[left,midway]{$\correz{\mu_1} Z$};

    \draw[dashed,-{Latex[length=2.mm, width=1.5mm]},thick] (m)--(s);

        \draw[dashed,-{Latex[length=2.mm, width=1.5mm]},thick] (m)--(i);

            \draw[dashed,-{Latex[length=2.mm, width=1.5mm]},thick] (z)--(s);

        \draw[dashed,{Latex[length=2.mm, width=1.5mm]}-{Latex[length=2.mm, width=1.5mm]},thick] (z)--(i);

        \draw[dashed,-{Latex[length=2.mm, width=1.5mm]},thick] (i)--(u);

		\end{tikzpicture}
			\caption{Flow diagram for system \eqref{mod_3}. Solid lines represent movement of individuals between compartments, dashed lines represent influences between information and epidemic spreading. Notice that the only double arrow connects $Z$ and $I$, which is the only mutual influence between the information and epidemic layers. Red: fast variables; blue: slow variables, where a multiplication by $\varepsilon$ is implicit on all arrows, for ease of notation.}
			\label{fig:flow}
		\end{figure}
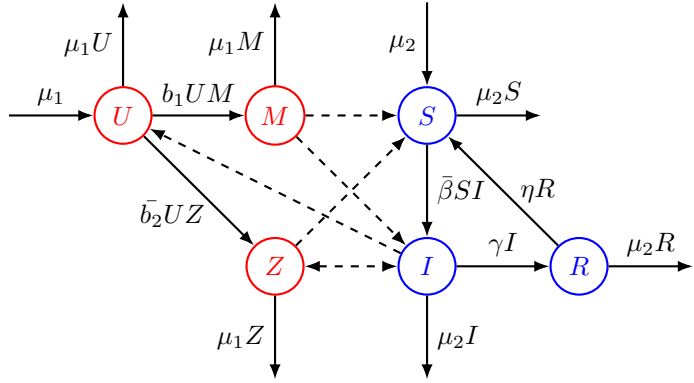

\section{Fast system: equilibria and stability}\label{sec:fast}

In this section, we focus on local and global stability of the equilibria of system \eqref{mod_3}, depending on the parameters involved. In particular, we remark that, when $\varepsilon=0$, the slow variables $S$, $I$ and $R$ do not evolve, and can (and will) be considered as parameters for the layer equations.

\subsection{Misinformed-Skeptical-Free Equilibrium}

We consider the fast system \eqref{mod_3} fixing $\varepsilon = 0$ and we compute the Basic Reproduction Number $\R0$ using the next generation matrix method, see \cite{Diekman1990}. The Misinformed-Skeptical-Free Equilibrium (MSFE) of \eqref{mod_3_U}--\eqref{mod_3_Z}, i.e. the equilibrium in which $M = Z = 0$, can be computed
as \begin{equation}\label{MSFE_eq}
    \tilde{E}_0 = (U_0, M_0, Z_0) = (1, 0, 0)
\end{equation} from easy calculations. We remark that both $M$ and $Z$ ``recruit'' individuals from the information-susceptible compartment $U$; hence, we consider both for the derivation of $\mathcal{R}_0^f$.

\begin{lemma}
The Basic Reproduction Number of system \eqref{mod_3_U}--\eqref{mod_3_Z} is $$\R0^{\text{f}}= \max\left\lbrace \dfrac{b_1}{\correz{\mu_1}},  \dfrac{b_2}{\correz{\mu_1}(1-K I)} \right\rbrace.$$
\end{lemma}
\begin{proof}
The ``disease'' (information) compartments of the model are 
\begin{subequations}\label{info_eqs}
\begin{align}
\frac{\text{d} M}{\text{d} t}&= b_1 U M -\correz{\mu_1} M,  \\
\frac{\text{d}Z}{\text{d} t}&=\dfrac{b_2}{1-KI} UZ-\correz{\mu_1} Z ,
\end{align}
\end{subequations}
where $I \in [0,1]$ is a parameter indicating the initial fraction of infected individuals in the population. We define
\begin{align*}
     \mathcal{F}_1(U,M,Z) \coloneqq b_1 U M, \qquad  \mathcal{V}_1(U,M,Z)\coloneqq \correz{\mu_1} M, \\
     \mathcal{F}_2(U,M,Z) \coloneqq \dfrac{b_2}{1-KI} UZ, \qquad  \mathcal{V}_2(U,M,Z)\coloneqq \correz{\mu_1} Z.
\end{align*}
We then obtain
\begin{equation*}
   F= \left( \begin{matrix} \label{mat}
\dfrac{\partial \mathcal{F}_1}{\partial M}(\tilde{E}_0) & \dfrac{\partial \mathcal{F}_1}{\partial Z}(\tilde{E}_0) \\ \\
\dfrac{\partial \mathcal{F}_2}{\partial M}(\tilde{E}_0) & \dfrac{\partial \mathcal{F}_2}{\partial Z}(\tilde{E}_0)
\end{matrix} \right) = \left(\begin{matrix}
        b_1  & 0 \\
        0 & \dfrac{b_2}{1-K I} 
    \end{matrix}\right)
    \qquad \text{and} \qquad  V = \left( \begin{matrix} 
\dfrac{\partial \mathcal{V}_1}{\partial M}(\tilde{E}_0) & \dfrac{\partial \mathcal{V}_1}{\partial Z}(\tilde{E}_0) \\ \\
\dfrac{\partial \mathcal{V}_2}{\partial M}(\tilde{E}_0) & \dfrac{\partial \mathcal{V}_2}{\partial Z}(\tilde{E}_0)
\end{matrix} \right) = \left(\begin{matrix}
        \correz{\mu_1} & 0 \\
        0 & \correz{\mu_1}
    \end{matrix}\right).
    \end{equation*}
Therefore, the next generation matrix, defined as $W\coloneqq F V^{-1}$, is
\begin{equation*}
    W  = \left(\begin{matrix}
        b_1  & 0 \\
        0 &  \dfrac{b_2}{1-K I} 
    \end{matrix}\right)\left(\begin{matrix}
        \frac{1}{\correz{\mu_1}} & 0 \\ \\
        0 & \frac{1}{\correz{\mu_1}}
    \end{matrix}\right) \quad \Longrightarrow \quad \R0^{\text{f}}\coloneqq \rho(W) = \max\left\lbrace \dfrac{b_1}{\correz{\mu_1}},  \dfrac{b_2}{\correz{\mu_1}(1-K I)} \right\rbrace,
\end{equation*}
where $\rho(\cdot)$ denotes the spectral radius of a matrix.
\end{proof}

We now prove two results on the local and global stability of the MSFE \eqref{MSFE_eq}, respectively.

\begin{lemma}
[Local stability of the MSFE]\label{local_MSFE}
The MSFE is locally asymptotically stable for the subsystem \eqref{mod_3_U}--\eqref{mod_3_Z} if $\R0^{\text{f}}<1$, and unstable if $\R0^{\text{f}} >1$.
\end{lemma}
\begin{proof}
    The Jacobian matrix of \eqref{mod_3_U}--\eqref{mod_3_Z} computed at the MSFE $\tilde{E}_0$ is
\begin{equation*}\label{eqn:Jac_matrix_E0}
   J_{|\tilde{E}_0}= \begin{pmatrix}
    -\correz{\mu_1} & - b_1 & -\dfrac{b_2}{1-KI}  & \\
    0 & b_1 - \correz{\mu_1} & 0 \\
    0 & 0 & \dfrac{b_2}{1-KI}-\correz{\mu_1}  \\
    \end{pmatrix}.
\end{equation*}
The eigenvalues of $J_{|\tilde{E}_0}$ can be easily computed as 
\begin{equation*}
    \lambda_1 = -\correz{\mu_1} , \qquad \lambda_2 = b_1-\correz{\mu_1} , \qquad \lambda_3 =  \dfrac{b_2}{1-KI}-\correz{\mu_1}.
\end{equation*}
Thus, if $\R0^{\text{f}}<1$ all the eigenvalues are negative and the MSFE is locally asymptotically stable. If, instead, $\R0^{\text{f}} >1$, one (or both) of $\lambda_2$ or $\lambda_3$ is positive, and the MSFE loses local stability.
\end{proof}

\begin{restatable}{proposition}{globalMSFE}
\label{global_MSFE}
   Assume $\R0^{\text{f}}<1$. Then, the MSFE is exponentially stable for the subsystem \eqref{mod_3_U}--\eqref{mod_3_Z}.
\end{restatable}
\begin{proof}
\correz{The proof of Proposition \ref{global_MSFE} can be found in Appendix \ref{app_gs}.}    
\end{proof}

\subsection{Skeptical-Free and Misinformed-Free Equilibria}
We now discuss the existence of other equilibria of system \eqref{mod_3_U}--\eqref{mod_3_Z}, i.e. the equilibria in which $M\neq 0$ and/or $Z\neq  0$.

\begin{theorem}\label{thm:equilibria}
    Assume $\R0^{\text{f}}>1$. We distinguish the following cases.
\begin{itemize}
    \item If $\dfrac{b_1}{\correz{\mu_1}} > 1$ and $\dfrac{b_2}{\correz{\mu_1} (1-KI)} < 1$ then system \eqref{mod_3_U}--\eqref{mod_3_Z} has a unique positive equilibrium \begin{equation}\label{eq:SFE}
        \tilde{E}_1 = \left( \dfrac{\correz{\mu_1}}{b_1}, \dfrac{b_1-\correz{\mu_1}}{b_1},0\right),
    \end{equation} called Skeptical-Free Equilibrium (SFE).
    \item If $\dfrac{b_1}{\correz{\mu_1}} < 1$ and $\dfrac{b_2}{\correz{\mu_1} (1-KI)} > 1$ then system \eqref{mod_3_U}--\eqref{mod_3_Z} has a unique positive equilibrium 
    \begin{equation}\label{eq:MFE}
     \tilde{E}_2 = \left( \dfrac{\correz{\mu_1}(1-KI)}{b_2}, 0,\dfrac{b_2-\correz{\mu_1}(1-KI)}{b_2}\right),   
    \end{equation}
    called Misinformed-Free Equilibrium (MFE).
    \item If $\dfrac{b_1}{\correz{\mu_1}} > 1$ and $\dfrac{b_2}{\correz{\mu_1} (1-KI)} > 1$ then \eqref{mod_3_U}--\eqref{mod_3_Z} admits two positive equilibria, namely the SFE and the MFE defined above.
\end{itemize}
\end{theorem}
\begin{proof}
In order to compute the equilibria of system \eqref{mod_3_U}--\eqref{mod_3_Z} we need to solve the following system of equations 
\begin{subequations}
  \begin{numcases}{}
    \correz{\mu_1} -\left(b_1 M +  \dfrac{b_2}{1-KI} Z + \correz{\mu_1}\right)  U = 0, \label{U_zero}\\ 
    (b_1 U - \correz{\mu_1} ) M = 0, \label{M_zero}\\ 
    \left( \dfrac{b_2}{1-KI} U - \correz{\mu_1} \right) Z = 0. \label{Z_zero}
\end{numcases}
\end{subequations} 

From \eqref{M_zero} and \eqref{Z_zero} we obtain
\begin{equation}\label{eq_slt}
 \left\lbrace   U = \dfrac{\correz{\mu_1}}{b_1} \quad \text{or}\quad  M =0\right\rbrace \qquad \text{and} \qquad \left\lbrace U = \dfrac{\correz{\mu_1}(1-KI)}{b_2} \quad \text{or}\quad  Z =0  \right\rbrace.
\end{equation}
Since $U$ can assume either of the two values, we distinguish between the following two cases.

\begin{itemize}
    \item $U =\dfrac{\correz{\mu_1}}{b_1}$ and $Z=0$. Substituting in \eqref{U_zero} we obtain
    \begin{equation*}
      \correz{\mu_1} - \left(b_1 M + \correz{\mu_1} \right) \dfrac{\correz{\mu_1}}{b_1} = 0 \quad \Longrightarrow \quad  M = \dfrac{b_1 - \correz{\mu_1}}{b_1},
     \end{equation*}
     which is feasible if and only if $b_1 - \correz{\mu_1} > 0$.
    \item $U = \dfrac{\correz{\mu_1}(1-KI)}{b_2} $ and $M = 0 $. Substituting in \eqref{U_zero} we obtain
    \begin{equation*}
        \correz{\mu_1} - \left( \dfrac{b_2}{\correz{\mu_1}(1-KI)} Z + \correz{\mu_1}\right) \dfrac{\correz{\mu_1}(1-KI)}{b_2} = 0  \quad \Longrightarrow \quad Z = \dfrac{b_2 - \correz{\mu_1} (1-KI)}{b_2},
    \end{equation*}
    which is feasible if and only if $b_2 - \correz{\mu_1}(1-KI) > 0$.
\end{itemize}

Thus we have obtained the following two equilibria:
\begin{align*}
    \tilde{E}_1 \coloneqq (U_1,M_1,Z_1) = \left( \dfrac{\correz{\mu_1}}{b_1}, \dfrac{b_1-\correz{\mu_1}}{b_1},0\right),& \quad \text{which exists if } b_1 > \correz{\mu_1},\\ 
    \tilde{E}_2 \coloneqq (U_2, M_2,Z_2) = \left(\dfrac{\correz{\mu_1}(1-KI)}{b_2}, 0, \dfrac{b_2 - \correz{\mu_1}(1-KI)}{b_2} \right),& \quad  \text{which exists if } b_2 > \correz{\mu_1}(1-KI).
\end{align*}
Recall the expression of the Basic Reproduction Number
\begin{equation*}
    \R0^\text{f} =\max\left\lbrace \dfrac{b_1}{\correz{\mu_1}},  \dfrac{b_2}{\correz{\mu_1}(1-K I)} \right\rbrace,
\end{equation*}
we obtain that if $\R0^\text{f} > 1$ either only one between $\tilde{E}_1$ and $\tilde{E}_2$ exists, or both of them exist.
\end{proof}

We now investigate the local stability of the positive equilibria, studying the eigenvalues of the Jacobian matrix of system \eqref{mod_3_U}-\eqref{mod_3_Z} evaluated at each point.

\begin{lemma}[Local stability of the SFE]\label{local_SFE}
   Assume that $\dfrac{b_1}{\correz{\mu_1}} >1$. The SFE \eqref{eq:SFE} is locally asymptotically stable for the subsystem \eqref{mod_3_U}-\eqref{mod_3_Z} if $b_1 >\dfrac{b_2}{1-KI}$, and unstable if $b_1 <\dfrac{b_2}{1-KI}$.
\end{lemma}
\begin{proof}
    The Jacobian matrix of \eqref{mod_3_U}--\eqref{mod_3_Z} computed at the SFE $\tilde{E}_1$ is
\begin{equation*}\label{eqn:Jac_matrix_E1}
   J_{|\tilde{E}_1}= \begin{pmatrix}
    -b_1& -\correz{\mu_1} &-\dfrac{\correz{\mu_1} b_2}{b_1(1-KI)}  \\
    b_1-\correz{\mu_1}  & 0 & 0 \\
    0 & 0 & \correz{\mu_1}\left(\dfrac{b_2}{b_1(1-KI)}-1\right)\\
    \end{pmatrix}
\end{equation*}
The eigenvalues of $J_{|\tilde{E}_1}$ can be easily computed as 
\begin{equation*}
    \lambda_1 = -\correz{\mu_1} , \qquad \lambda_2 = -b_1+\correz{\mu_1} , \qquad \lambda_3 =  \correz{\mu_1}\left(\dfrac{b_2}{b_1(1-KI)}-1\right).
\end{equation*}
Recall that $\dfrac{b_1}{\correz{\mu_1}} >1$, which implies $\lambda_2 <0$. Moreover, if $b_1 >\dfrac{b_2}{1-KI}$ all the eigenvalues are negative and thus the SFE is locally asymptotically stable. If, instead,  $b_1 <\dfrac{b_2}{1-KI}$ then $\lambda_3 >0 $ and the SFE loses local stability. 
\end{proof}

\begin{lemma}[Local stability of the MFE]\label{local_MFE}
       Assume that $\dfrac{b_2}{\correz{\mu_1}(1-KI)} >1$. The MFE \eqref{eq:MFE} is locally asymptotically stable for the subsystem \eqref{mod_3_U}-\eqref{mod_3_Z} if $b_1 <\dfrac{b_2}{1-KI}$, and unstable if $b_1 >\dfrac{b_2}{1-KI}$.
\end{lemma}
\begin{proof}
    The Jacobian matrix of \eqref{mod_3_U}--\eqref{mod_3_Z} computed at the MFE $\tilde{E}_2$ is
\begin{equation*}\label{eqn:Jac_matrix_E2}
   J_{|\tilde{E}_2}= \begin{pmatrix}
  -  \dfrac{b_2}{1-KI}& - \dfrac{\correz{\mu_1} b_1(1-KI)}{b_2} &  -\correz{\mu_1} \\
    0 & \dfrac{b_1 \correz{\mu_1}(1-KI)}{b_2} -\correz{\mu_1}& 0 \\
    \dfrac{b_2-\correz{\mu_1}(1-KI)}{1-KI} & 0 & 0\\
    \end{pmatrix}.
\end{equation*}
The eigenavalues of $J_{|\tilde{E}_2}$ can be easily computed as 
\begin{equation*}
    \lambda_1 = -\correz{\mu_1} , \qquad \lambda_2 = \dfrac{\correz{\mu_1}(1-KI)-b_2}{1-KI} , \qquad \lambda_3 =  \correz{\mu_1}\left(\dfrac{b_1(1-KI)}{b_2} -1\right).
\end{equation*}

Recall that $\dfrac{b_2}{\correz{\mu_1}(1-KI)} >1$, which implies $\lambda_2 <0$. Moreover, if $b_1 <\dfrac{b_2}{1-KI}$ all the eigenvalues are negative and thus the SFE is locally asymptotically stable. If, instead,  $b_1 >\dfrac{b_2}{1-KI}$ then $\lambda_3 >0 $ and the MFE loses local stability. 
\end{proof}

\begin{remark}\label{remark_stab}
    Notice that if only $\tilde{E}_1$ lies in the biologically relevant region and $\tilde{E}_2$ does not, i.e. $\dfrac{b_1}{\correz{\mu_1}}>1$ and $\dfrac{b_2}{\correz{\mu_1} (1-KI)}<1$ then the condition for local asymptotically stability is always satisfied. Indeed, 
    \begin{equation*}
        b_1 > \correz{\mu_1} > \dfrac{b_2}{1-KI}.
    \end{equation*}
    Furthermore, we can conclude that if only $\tilde{E}_1$ exists it is always locally asymptotically stable. Similarly, when only $\tilde{E}_2$ exists, it is locally asymptotically stable.
\end{remark}

From Remark \ref{remark_stab}, we know that if only one of the two equilibria exists, then it is always locally asymptotically stable. We now
provide results on
the global stability of \eqref{eq:SFE} and \eqref{eq:MFE}, assuming the other equilibrium does not lie in the biologically relevant region. Although a result on global stability is stronger than the one on local stability, the exact formulation of the eigenvalues given in our local stability results will prove useful in the slow-fast analysis of the whole system \eqref{mod_3}. The proofs of Propositions \ref{global_SFE}-\ref{global_MFE2}, introduced below, are give in the Appendix \ref{app_gs}. We make extensive use of the classic Goh-Lotka-Volterra Lyapunov function \cite{cangiotti2023survey} in all cases.

\begin{restatable}[Global stability of the SFE]{proposition}{globalSFE}
\label{global_SFE}
   Assume that $\dfrac{b_1}{\correz{\mu_1}} >1$ and $\dfrac{b_2}{\correz{\mu_1} (1-KI)} < 1$. Then, the SFE \eqref{eq:SFE} is Globally Asymptotically Stable (GAS) for the subsystem \eqref{mod_3_U}-\eqref{mod_3_Z}.
\end{restatable}
\begin{proof}
\correz{The proof of Proposition \ref{global_SFE}  can be found in Appendix \ref{app_gs}.}    
\end{proof}

\begin{restatable}[Global stability of the MFE]{proposition}{globalMFE}
\label{global_MFE}
   Assume that $\dfrac{b_1}{\correz{\mu_1}} <1$ and $\dfrac{b_2}{\correz{\mu_1} (1-KI)} > 1$. Then, the MFE \eqref{eq:MFE} is GAS for the subsystem \eqref{mod_3_U}-\eqref{mod_3_Z}.
\end{restatable}
\begin{proof}
\correz{The proof of Proposition \ref{global_MFE}  can be found in Appendix \ref{app_gs}.}    
\end{proof}

\begin{remark}\label{remark_unstability}
    If both the positive equilibria exists, i.e. $b_1 >\correz{\mu_1}$ and $b_2> \correz{\mu_1}(1-KI)$, then Lemmas \ref{local_SFE} and \ref{local_MFE} ensure that the system can not exhibit bistability. Indeed, if $b_1 <\dfrac{b_2}{1-KI}$ then the SFE is unstable and if $b_1 >\dfrac{b_2}{1-KI}$ then the MFE is unstable. 
\end{remark}

In the following, we provide conditions for the global asymptotical stability of the positive equilibria also in the case in which both lie in the biologically relevant region, under the same hypothesis of Lemmas \ref{local_SFE} and \ref{local_MFE}.

\begin{restatable}[Global stability of the SFE]{proposition}{globalSFEb}
\label{global_SFE2}
   Assume that $\dfrac{b_1}{\correz{\mu_1}} >1$ and $\dfrac{b_2}{\correz{\mu_1} (1-KI)} > 1$. Then, the SFE \eqref{eq:SFE} is GAS for the subsystem \eqref{mod_3_U}-\eqref{mod_3_Z} if $b_1 >\dfrac{b_2}{1-KI}$.
\end{restatable}
\begin{proof}
\correz{The proof of Proposition \ref{global_SFE2} can be found in Appendix \ref{app_gs}.}    
\end{proof}

\begin{restatable}[Global stability of the MFE]{proposition}{globalMFEb}
\label{global_MFE2}
   Assume that $\dfrac{b_1}{\correz{\mu_1}} >1$ and $\dfrac{b_2}{\correz{\mu_1} (1-KI)} > 1$. Then, the MFE \eqref{eq:MFE} is GAS for the subsystem \eqref{mod_3_U}-\eqref{mod_3_Z} if $b_1 <\dfrac{b_2}{1-KI}$.
\end{restatable}
\begin{proof}
\correz{The proof of Proposition \ref{global_MFE2} can be found in Appendix \ref{app_gs}.}    
\end{proof}

\begin{remark}\label{remark_meaning_stab}
 Considering the Skeptical-Free Equilibrium, $\tilde{E}_1$, we observe that the eigenvalue $\lambda_3$ vanishes when
  $$ b_1 ^{\dag} = \dfrac{b_2}{1-KI} $$ 
  e.g. when the rate of becoming misinformed is equal to the rate of becoming skeptical. Moreover, the stability of either the Skeptical-Free Equilibrium (resp. Misinformed-Free Equilibrium) corresponds with the rate of the flow from $U$ to $M$ being smaller (resp. greater) than the rate of the flow from $U$ to $Z$, which is quite an intuitive result.  In the following Section we analyze the behaviour of the model under this condition.
\end{remark}

\subsection{Threshold case}\label{sec:threshold}

Note that in \eqref{eq_slt} we did not consider the case in which \textit{both} $M\neq 0$ and $Z\neq 0$. Indeed, in this case we need 
$$\dfrac{\correz{\mu_1}}{b_1} = \dfrac{\correz{\mu_1}(1-KI)}{b_2} \quad \Longleftrightarrow \quad b_1 = \dfrac{b_2}{1-KI},$$
which could be possible for a specific value of the slow variable $I$, namely
$$
I=\dfrac{b_1-b_2}{K b_1}.
$$

Notice that in this case the condition  $b_1 = \dfrac{b_2}{1-KI}$ is exactly the ``critical'' condition of Remark \ref{remark_meaning_stab}. 

Defining $b \coloneqq   b_1 = \dfrac{b_2}{1-KI}$, equations \eqref{mod_3_U}-\eqref{mod_3_Z} become 
\begin{subequations}\label{sum_model}
\begin{eqnarray}
\frac{\textup{d}U}{\textup{d} t}&=& \correz{\mu_1}  -b U (M+Z)-\correz{\mu_1} U  ,\\ 
\frac{\textup{d}M}{\textup{d}t}&=& b U M-\correz{\mu_1} M,\\
\frac{\textup{d}Z}{\textup{d}t}&=& b UZ-\correz{\mu_1} Z.
\end{eqnarray}
\end{subequations}
The value of the Basic Reproduction Number is simply $\R0^{\text{f}} = \dfrac{b}{\correz{\mu_1}}$ and the MSFE $\tilde{E}_0= (1,0,0)$ always exists and it is GAS if $\R0^{\text{f}}<1$. We remark that, from system \eqref{sum_model}, we can not find the explicit value of the positive equilibria; indeed, computing its expression \correz{(and recalling that, under our simplifying assumption $N=1$, $M+Z = 1-U$)} we obtain 
\begin{equation}\label{eq:caso_lim}
    U^* = \dfrac{1}{\R0^{\text{f}}} \quad \text{and} \quad M^*  = 1-Z^*- U^*.
\end{equation}
We thus have a line of equilibria for the fast system \eqref{sum_model}, $\tilde{E}_3 = (U^*,M^*,Z^*)$, which exists if $\R0^{\text{f}}>1$.  On this line, there exists only one equilibrium for the perturbed system \eqref{mod_3}, as we show in the Appendix \ref{app_eq}. The eigenvalues on this one dimensional set of equilibria for the fast flow are
$$
\lambda_1=0, \quad \lambda_{2,3}=\dfrac{-b-\correz{\mu_1} \pm \sqrt{(b-\correz{\mu_1})^2+4\correz{\mu_1}^2}}{2}.
$$
The zero eigenvalues $\lambda_1$ corresponds to the direction spanned by this line of equilibria. Notice that $\lambda_3$ ($-$ sign in front of the square root) it is always negative, while $\lambda_2$ ($+$ sign) is negative if $\R0^{\text{f}}>1$ (which is a condition needed for the feasibility of this equilibrium), and positive otherwise.

It is convenient to introduce an auxiliary variable, taking into account the additional symmetry this scenario introduces in the model, system
\eqref{sum_model} can be rewritten, introducing for ease of notation $L \coloneqq M+Z$, obtaining
\begin{subequations}\label{pat_sys}
\begin{eqnarray}
\frac{\textup{d}U}{\textup{d} t}&=& \correz{\mu_1} \correz{L}  -b U L  ,\\ 
\frac{\textup{d}L}{\textup{d}t}&=& b U L -\correz{\mu_1} L.
\end{eqnarray}
\end{subequations}
The unique positive equilibrium of system \eqref{pat_sys} is given by $\tilde{E}= \left(\dfrac{1}{\R0^{\text{f}}},1-\dfrac{1}{\R0^{\text{f}}}\right)$ and it exists in the biologically feasible region when $\R0^{\text{f}}>1$.

\begin{restatable}[Global stability of the positive equilibrium]{proposition}{globalEEpat}
\label{globalEEpat}
   Assume that $\R0^{\text{f}}>1$. Then, the positive equilibrium $\tilde{E}= \left(\dfrac{1}{\R0^{\text{f}}},1-\dfrac{1}{\R0^{\text{f}}}\right)$ is globally asymptotically stable for the subsystem \eqref{pat_sys}.
\end{restatable}
\correz{
\begin{proof}
The proof of Proposition \ref{globalEEpat} can be found in Appendix \ref{app_gs}.
\end{proof}
}
\begin{remark}
    The set of equilibria \eqref{eq:caso_lim} is not a branch of the critical manifold, since its existence depends on $I$, one for the slow variables, assuming a specific, fixed value. Hence, the slow flow can not evolve on this set.
\end{remark}

\subsection{Bifurcation analysis}

The fast system \eqref{mod_1} admits four equilibrium points. We summarize the respective feasibility and stability conditions in Table \eqref{tab:tabella}. Next, we prove the existence of three transcritical bifurcations between $\tilde{E}_0$ and $\tilde{E}_1$, $\tilde{E}_2$ or $\tilde{E}_3$, respectively.
\begin{table}[ht]
\centering
{\renewcommand{\arraystretch}{2.8}
\begin{tabular}{|l|l|l|}
\hline \textbf{Equilibria} & \textbf{Feasibility} &\textbf{Stability} \\
\hline $\tilde{E}_0 = (1,0,0)$ & Always feasible & $\R0^{\text{f}}= \max\left\lbrace \dfrac{b_1}{\correz{\mu_1}},  \dfrac{b_2}{\correz{\mu_1}(1-K I)} \right\rbrace<1$ \\
\hline $\tilde{E}_1 = \left( \dfrac{\correz{\mu_1}}{b_1}, \dfrac{b_1-\correz{\mu_1}}{b_1},0\right)$ & $\dfrac{b_1}{\correz{\mu_1}} > 1$  & $b_1 >\dfrac{b_2}{1-KI}$ \\
\hline
$\tilde{E}_2  = \left(\dfrac{\correz{\mu_1}(1-KI)}{b_2}, 0, \dfrac{b_2 - \correz{\mu_1}(1-KI)}{b_2} \right)$  & $\dfrac{b_2}{\correz{\mu_1} (1-KI)}> 1$    & $b_1 <\dfrac{b_2}{1-KI}$  \\
\hline $\tilde{E}_3=(U^*,M^*, Z^*)$  & $ b:=b_1 = \dfrac{b_2}{1-KI}$ and $\R0^{\text{f}}>1$ & Always when feasible  \\ \hline
\end{tabular}}
\caption{Summary of the equilibria of the fast subsystem, \eqref{mod_1}, their feasibility and stability conditions.\label{tab:tabella}}
\end{table}

\begin{restatable}{proposition}{TB_E0_E1}
\label{TB_E0_E1}
  System \eqref{mod_1} exhibits a transcritical bifurcation at $b_1^*:=b_1 = \correz{\mu_1}$ between $\tilde{E}_0$ and $\tilde{E}_1$.
\end{restatable}
 
\begin{proof} 
Using Sotomayor's theorem \cite[Ch. 4, Thm. 1,]{perko}, we can prove that there is a transcritical bifurcation for which $\tilde{E}_1$ emanates from $\tilde{E}_0$ as soon as $b_1$ increases past $\correz{\mu_1}$. Focusing on $\tilde{E}_0$ and the parameter $b_1$, we can observe that the first and the third eigenvalues of \eqref{eqn:Jac_matrix_E0} are always negative, while the second one vanishes if $b_1^*:=b_1 = \correz{\mu_1}$. In this case, right and left eigenvectors corresponding to the zero eigenvalue are given by $ v =\left[1,-1,0 \right]$ and $w= \left[0,1,0 \right]$.

We denote with $F(U,M,Z)= \left[F^1,F^2,F^3 \right]^T$ system's \eqref{mod_1} right-hand side. Since 
\begin{equation*}
    F_{b_1} = [-UM, UM, 0]^T  \quad \mbox{and} \quad DF_{b_1} = \begin{bmatrix}
    -M & -U & 0 \\
    M & U & 0 \\
    0 & 0 & 0
\end{bmatrix} \,,
\end{equation*}
where $F_{b_1}$ indicates the derivative of $F$ with respect to $b_1$, and $DF_{b_1}$ the Jacobian matrix of this derivative, it follows that
\begin{equation*}
  w^T F_{b_1} (\tilde{E}_0, b_1^*) = 0  \quad \mbox{and} \quad w^T \left[ DF_{b_1} (\tilde{E}_0, b_1^*) \, v \right] = - 1 \neq 0 \,.
\end{equation*}
Thus, the first two sufficient conditions of Sotomayor's theorem to ensure the existence of a transcritical bifurcation from $\tilde{E}_0$ at $b_1 = b_1^*$ are verified. We observe that the second partial derivatives of $F^2$ are $F^2_{UU} = F^2_{MM} = F^2_{ZZ} = F^2_{UZ} = F^2_{MZ} =0 $ and $F^2_{UM} = b_1$ and thus at this equilibrium all second derivatives vanish, except for $F^2_{UM}$. In summary, we find
\begin{equation*}
  w^T \left[D^2 F (\tilde{E}_0,b_1^*) (v,v)\right] = D^2 F^2 (\tilde{E}_0,b_1^*)(v,v) 
  = -\correz{\mu_1} \neq 0
\end{equation*}
so that also the third condition  of Sotomayor's theorem for the occurrence of a transcritical bifurcation is satisfied. 
\end{proof}

\begin{restatable}{proposition}{TB_E0_E2}
\label{TB_E0_E2}
  System \eqref{mod_1} exhibits a transcritical bifurcation at $b_2^*:=b_2 = \correz{\mu_1} (1-IK)$ between $\tilde{E}_0$ and $\tilde{E}_2$ .
\end{restatable}
 
\begin{proof} 
Using Sotomayor's theorem once again, we prove that there is a transcritical bifurcation for which $\tilde{E}_2$ emanates from $\tilde{E}_0$ as soon as $b_2$ increases past $\correz{\mu_1}(1-KI)$. Focusing on $\tilde{E}_0$ and the parameter $b_2$, we can observe that the first and the second eigenvalues of \eqref{eqn:Jac_matrix_E0} are always negative, while the third one vanishes if $b_2^*:=b_2 = \correz{\mu_1}(1-KI)$. In this case, right and left eigenvectors corresponding to the zero eigenvalue are given by $ v =\left[1,0, -1 \right]$ and $w= \left[0,0,1 \right]$.

We denote by $F(U,M,Z)= \left[F^1,F^2,F^3 \right]^T$ the system's \eqref{mod_1} right-hand side. Since 
\begin{equation*}
    F_{b_2} = \left[-\dfrac{UZ}{1-KI}, 0, \dfrac{UZ}{1-KI}\right]^T  \quad \mbox{and} \quad DF_{b_2} = \begin{bmatrix}
    -\dfrac{Z}{1-KI} & 0 & -\dfrac{U}{1-KI}\\
    0 & 0 & 0 \\
    \dfrac{Z}{1-KI} & 0 & \dfrac{U}{1-KI}
\end{bmatrix} \,,
\end{equation*}
where $F_{b_2}$ indicates the derivative of $F$ with respect to $b_2$, and $DF_{b_2}$ the Jacobian matrix of this derivative, it follows that
\begin{equation*}
  w^T F_{b_2} (\tilde{E}_0, b_2^*) = 0  \quad \mbox{and} \quad w^T \left[ DF_{b_2} (\tilde{E}_0, b_2^*) \, v \right] = - \dfrac{1}{1-KI} \neq 0 \,.
\end{equation*}
Thus, the first two sufficient conditions of Sotomayor's theorem to ensure the existence of a transcritical bifurcation from $\tilde{E}_0$ at $b_2 = b_2^*$ are verified. We observe that the second partial derivatives of $F^3$ are $F^3_{UU} = F^3_{MM} = F^3_{ZZ} = F^3_{UM} = F^3_{MZ} =0 $ and 
$$F^3_{UZ} = \dfrac{b_2}{1-KI},$$
and thus at this equilibrium all vanish, except for $F^3_{UZ}$. In summary, we find
\begin{equation*}
  w^T \left[D^2 F (\tilde{E}_0,b_2^*) (v,v)\right] = D^2 F^3 (\tilde{E}_0,b_2^*)(v,v) 
  = -\correz{\mu_1} \neq 0
\end{equation*}
so that also the third condition of Sotomayor's theorem for the occurrence of a transcritical bifurcation is satisfied. 
\end{proof}

\begin{restatable}{proposition}{TB_E0_E3}
\label{TB_E0_E3}
  Let $b := b_1 =  \dfrac{b_2}{\correz{\mu_1}(1-KI)}$ hold. Then, system \eqref{mod_1} exhibits a transcritical bifurcation at $b^* = \correz{\mu_1}$ between $\tilde{E}_0$ and $\tilde{E}_3$.
\end{restatable}
 
\begin{proof} 
The claim follows as a particular case of Propositions \ref{TB_E0_E1} and \ref{TB_E0_E2}. 
\end{proof}

\section{Slow systems: equilibria and stability} \label{sec:slow}

In singularly perturbed system of ODEs in standard form, such as system \eqref{mod_3}, the set of equilibria for the fast variable plays a fundamental role in the dynamics. Such set is called the \textit{critical manifold}.

As a consequence of Theorem \ref{thm:equilibria}, the critical manifold of system \eqref{mod_3} is given by the union of the following three sets, which exists in the biologically relevant region according to the conditions outlined in Theorem \ref{thm:equilibria}:
\begin{equation}\label{eq:crit_manif}
  \begin{aligned} \mathcal{C}_0 = \mathcal{C}_0^0 \cup \mathcal{C}_0^1 \cup \mathcal{C}_0^2 :=& \{ (U,M,Z,S,I,R) \in \mathbb{R}^6 \; | \; (U,M,Z)=\Tilde{E}_0, \;S,I,R \geq 0, S+I+R=1 \} \;\cup\\
  & \{ (U,M,Z,S,I,R) \in \mathbb{R}^6 \; | \; (U,M,Z)=\Tilde{E}_1, \;S,I,R \geq 0, S+I+R=1 \} \;\cup\\
  &\{ (U,M,Z,S,I,R) \in \mathbb{R}^6 \; | \; (U,M,Z)=\Tilde{E}_2(I), \;S,I,R \geq 0, S+I+R=1 \} .      
  \end{aligned} 
\end{equation}

In this section, we study the behaviour of the slow system \eqref{mod_2} on each of the three branches of the critical manifold. Recall that the exact formulation of each equilibrium is given as $\tilde{E}_0$ \eqref{MSFE_eq}, $\tilde{E}_1$ \eqref{eq:SFE} and $\tilde{E}_2$ \eqref{eq:MFE} in the previous section. It is natural to study the evolution of the systems presented in this section in the slow time scale $\tau$. \correz{The results obtained in this section, when combined with the ones obtained in section \ref{sec:fast}, provide insight on the global behaviour of the perturbed system \eqref{mod_3} for $0<\varepsilon\ll 1$.}

\subsection{The first branch of the critical manifold, $\Tilde{E}_0$}\label{sec:lento_0}

On this branch, we have
$$
 \correz{\bar{\beta}(M,Z)}=\bar{\beta}(0,0)= \beta \dfrac{1+0}{1+0}=\beta.
$$
Hence, the slow system on $\mathcal{C}_0^0$ evolves according to the following system of ODEs:
\begin{subequations}\label{lento_0}
\begin{eqnarray}
\frac{\text{d}S}{\text{d} \tau}&=&  \correz{\mu_2}
   -\beta S I+\eta R- \correz{\mu_2} S,\label{lento_0_S}\\ 
\frac{\text{d}I}{\text{d} \tau}&=&  \beta S I- \gamma I-\correz{\mu_2} I ,\label{lento_0_I}   \\
\frac{\text{d}R}{\text{d} \tau}&=& \gamma I- \eta R-\correz{\mu_2} R\label{lento_0_R}, 
\end{eqnarray}
\end{subequations}
which is a classical SIRS model with demography and constant population. The corresponding Basic Reproduction Number is 
\begin{equation}\label{eq:rzerozero}
\mathcal{R}_0^{s,0}=\dfrac{\beta}{\gamma+\correz{\mu_2}}.
\end{equation}
It is well known \cite{o2010lyapunov} that if $\mathcal{R}_0^{s,0}<1$ the system \eqref{lento_0} converges to its Disease Free Equilibrium $(S,I,R)=(1,0,0)$, whereas if $\mathcal{R}_0^{s,0}>1$, there exists only a unique Endemic Equilibrium $E^{*,0} = (S^{*,0},I^{*,0},R^{*,0})$
\begin{equation}\label{EE_C00}
S^{*,0} = \dfrac{\gamma + \correz{\mu_2}}{\beta}, \qquad I^{*,0} = \dfrac{(\beta - \gamma - \correz{\mu_2})(\eta + \correz{\mu_2})}{\beta (\gamma + \eta + \correz{\mu_2})}, \qquad R^{*,0} = \dfrac{\gamma I^{*,0}}{\eta + \correz{\mu_2}},
\end{equation}
which is globally stable for system \eqref{lento_0}.

\subsection{The second branch of the critical manifold, $\Tilde{E}_1$}\label{sec:lento_1}

On this branch, we have
\begin{equation}\label{eq:betabarrE_1}
 \correz{\bar{\beta}(M,Z)}=\bar{\beta}\left(\dfrac{b_1-\correz{\mu_1}}{b_1},0\right)= \beta \dfrac{1+\dfrac{b_1-\correz{\mu_1}}{b_1}}{1+0}=\beta\dfrac{2b_1-\correz{\mu_1}}{b_1}.    
\end{equation}
Notice that, by the assumption on existence of $\tilde{E}_1$, $\correz{\mu_1}<b_1$, hence $ \correz{\bar{\beta}(M,Z)}>\beta$. From a biological point of view, this means that the absence of skeptical individuals increases the spread of the disease.

Then, the slow system on $\mathcal{C}_0^1$ evolves according to the following system of ODEs:
\begin{subequations}\label{lento_1}
\begin{eqnarray}
\frac{\text{d}S}{\text{d} \tau}&=&  \correz{\mu_2}
   -\correz{\beta\dfrac{2b_1-\correz{\mu_1}}{b_1}} S I+\eta R- \correz{\mu_2} S ,\label{lento_1_S}\\ 
\frac{\text{d}I}{\text{d} \tau}&=&  \correz{\beta\dfrac{2b_1-\correz{\mu_1}}{b_1}} S I- \gamma I-\correz{\mu_2} I ,\label{lento_1_I}   \\
\frac{\text{d}R}{\text{d} \tau}&=& \gamma I- \eta R-\correz{\mu_2} R \label{lento_1_R}, 
\end{eqnarray}
\end{subequations}
which is again a classical SIRS model with demography and constant population. The corresponding Basic Reproduction Number is 
\begin{equation}\label{eq:rzerouno}
\mathcal{R}_0^{s,1}=\dfrac{\beta}{\gamma+\correz{\mu_2}}\cdot\correz{\dfrac{2b_1-\correz{\mu_1}}{b_1}}.
\end{equation}
Furthermore, compared to the dynamics described in Section \ref{sec:lento_0}, due to the influence of misinformation spreading, represented by the parameter $b_1$, the Basic Reproduction Number is larger. In particular, one could imagine a situation in which $\mathcal{R}_0^{s,0}$ \eqref{eq:rzerozero} is smaller than 1, while $\mathcal{R}_0^{s,1}$ \eqref{eq:rzerouno} is larger than 1. This means that a disease which would naturally become extinct is ``kept alive'' by misinformed individual making it endemic, as illustrated in Figure \ref{fig:case_3}.

Similarly to the previous scenario, it is well known \cite{o2010lyapunov} that if $\mathcal{R}_0^{s,1}<1$ the system \eqref{lento_1} converges to its Disease Free Equilibrium $(S,I,R)=(1,0,0)$, whereas if $\mathcal{R}_0^{s,1}>1$, there exists a unique Endemic Equilibrium $E^{*,1} = (S^{*,1},I^{*,1},R^{*,1})$
\begin{equation}\label{EE_C01}\correz{
S^{*,1} = \dfrac{(\gamma + \correz{\mu_2} )b_1}{\beta (2 b_1 - \correz{\mu_1})}, \qquad I^{*,1} = \dfrac{(\beta (2b_1-\correz{\mu_1})-b_1(\gamma + \correz{\mu_2} ))(\eta + \correz{\mu_2})}{\beta (2b_1-\correz{\mu_1})(\gamma + \eta + \mu_2)}, \qquad R^{*,1} = \dfrac{\gamma I^{*,1}}{\eta + \mu_2},}
\end{equation}
which is globally asymptotically stable for system \eqref{lento_1}. 

\subsection{The third branch of the critical manifold, $\Tilde{E}_2$}\label{sec:lento_2}

On this branch, we have
$$
\correz{\bar{\beta}(M,Z)}=\bar{\beta}\left(0,\dfrac{b_2-\correz{\mu_1}(1-KI)}{b_2}\right)=\beta \dfrac{1+0}{1+\dfrac{b_2-\correz{\mu_1}(1-KI)}{b_2}}=\beta \dfrac{b_2}{2b_2-\correz{\mu_1}(1-KI)}.
$$
Notice that, by the feasibility assumption of $\tilde{E}_2$, $b_2>\correz{\mu_1}(1-KI)$, hence $\correz{\bar{\beta}(M,Z)}<\beta$. From a biological viewpoint, this means that the absence of misinformed individuals and the presence of skeptical individuals decrease the infection rate. This can be explained by the fact that aware individuals directly decrease the spread of the disease.

Hence, the slow system on $\mathcal{C}_0^2$ evolves according to the following system of ODEs:
\begin{subequations}\label{lento_2}
\begin{eqnarray}
\frac{\text{d}S}{\text{d} \tau}&=&  \correz{\mu_2}
   -\beta \dfrac{b_2}{2b_2-\correz{\mu_1}(1-KI)} S I+\eta R- \correz{\mu_2} S ,\label{lento_2_S}\\ 
\frac{\text{d}I}{\text{d} \tau}&=&  \beta \dfrac{b_2}{2b_2-\correz{\mu_1}(1-KI)} S I- \gamma I-\correz{\mu_2} I ,\label{lento_2_I}   \\
\frac{\text{d}R}{\text{d} \tau}&=& \gamma I- \eta R-\correz{\mu_2}R \label{lento_2_R}. 
\end{eqnarray}
\end{subequations}
Moreover, compared to the systems described in Sections \ref{sec:lento_0} and \ref{sec:lento_1}, the infected population plays a non-trivial role in the infection parameter. 

System \eqref{lento_2} always admits a Disease Free Equilibrium $(S,I,R)=(1,0,0)$. 
The disease compartment is $I$, thus we consider
\begin{equation*}
\frac{\text{d}I}{\text{d} \tau}= \beta \dfrac{b_2}{2b_2-\correz{\mu_1}(1-KI)} S I- \gamma I-\correz{\mu_2} I ,
\end{equation*}
where we define
\begin{equation*}
\mathcal{F} \coloneqq  \beta \dfrac{b_2}{2b_2-\correz{\mu_1}(1-KI)} S I \qquad \text{and}  \qquad \mathcal{V} \coloneqq (\gamma+\correz{\mu_2}) I.
\end{equation*}
We thus obtain
\begin{equation*}
    F = \dfrac{\partial \mathcal{F}}{\partial I} \bigg\rvert_{(S,I,R)=(1,0,0)} = \dfrac{\beta b_2}{2b_2 -\correz{\mu_1}} \qquad \text{and} \qquad V= \gamma+\correz{\mu_2},
\end{equation*}
from which 
\begin{equation}\label{eq:rzerodue}
    \mathcal{R}_0^{s,2} = \dfrac{\beta}{\gamma + \correz{\mu_2}} \cdot \dfrac{b_2}{2 b_2 - \correz{\mu_1}}.
\end{equation}
For system \eqref{lento_2} we can use the results from \cite{lahrouz2012complete} to ensure convergence to the DFE when $\mathcal{R}_0^{s,2}<1$, and when $\mathcal{R}_0^{s,2}>1$ there  exists a unique Endemic Equilibrium $E^{*,2} = (S^{*,2},I^{*,2},R^{*,2})$
\begin{equation}\label{EE_C02}
S^{*,2}= \dfrac{(\gamma + \correz{\mu_2}) ( 2b_2 - \correz{\mu_1}(1-KI^{*,2}))}{\beta b_2}, \quad I^{*,2} = \dfrac{(\gamma+\correz{\mu_2})(\beta b_2 - (\gamma+\correz{\mu_2})(2b_2-\correz{\mu_1}))}{\beta b_2 (\eta + \correz{\mu_2} + \gamma) +\correz{\mu_1} K ( \eta+\correz{\mu_2})(\gamma+\correz{\mu_2})}, \quad R^{*,2} = \dfrac{\gamma I^{*,2}}{\eta + \correz{\mu_2}},
\end{equation}
which is globally asymptotically stable for system \eqref{lento_2}. Notice that $\mathcal{R}_0^{s,2}>1$ if and only if $I^{*,2}>0$; moreover, after some calculations, one may check that $1-S^{*,2}=I^{*,2}+R^{*,2}$. Indeed, given $S^{*,2}$ and $R^{*,2}$ from equations \eqref{lento_2_I} and \eqref{lento_2_R}, equation \eqref{lento_2_S} becomes
$$
\correz{\mu_2}(1-S^{*,2}) - (\gamma + \correz{\mu_2})I^{*,2} + \dfrac{\eta \gamma I^{*,2}}{\eta + \correz{\mu_2}} = 0,
$$
and substituting $1-S^{*,2} = I^{*,2} + R^{*,2}$ we obtain
$$
\correz{\mu_2}\left(I^{*,2} + \dfrac{\gamma I^{*,2}}{\eta + \correz{\mu_2}}\right) - (\gamma + \correz{\mu_2})I^{*,2} + \dfrac{\eta \gamma I^{*,2}}{\eta + \correz{\mu_2}} = 0,$$
from which 
$$\correz{\mu_2} (\eta + \correz{\mu_2}) + \correz{\mu_2} \gamma - (\gamma + \correz{\mu_2}) (\eta + \correz{\mu_2}) + \eta \gamma = 0,
$$
which is a tautology.

To clarify the use of \cite[Thm. 3.1]{lahrouz2012complete}, compared to their notation we have: $b=\correz{\mu_2}$ (birth rate equal to mortality rate), $p=0$ (no vaccination rate), $c=0$ (no additional disease mortality), 
$$
\psi(I)=\dfrac{2b_2-\correz{\mu_1}(1-KI)}{2b_2-\correz{\mu_1}} \qquad \text{and} \qquad \beta=\dfrac{\beta b_2}{2b_2-\correz{\mu_1}}.
$$
All the hypotheses needed to use \cite[Thm. 3.1]{lahrouz2012complete} are then satisfied, and we can conclude global stability of the DFE or of the EE depending on $\mathcal{R}_0^{s,2}\lessgtr 1$.

Mirroring the remark in Section \ref{sec:lento_1}, we notice here that a disease characterized by parameters for which $\mathcal{R}_0^{s,0}>1$ might be brought to extinction, meaning $\mathcal{R}_0^{s,2}<1$, if enough individuals are skeptical of disease-related misinformation. We illustrate this example in Figure \ref{fig:case_4}.

\section{Delayed loss of stability: the entry-exit function}\label{sec:entry-exit}

Depending on specific relations between parameters of the fast systems \eqref{mod_1} and parameters of the slow system \eqref{mod_2}\correz{, evolving on the three branches of the critical manifold \eqref{eq:crit_manif} as detailed in section \ref{sec:slow}}, two scenarios can happen in the perturbed system \eqref{mod_3}. In this section, we explore these two possibilities.

Recall Lemmas \ref{local_MSFE}, \ref{local_SFE}, \ref{local_MFE} and Table \ref{tab:tabella}. Due to the corresponding existence constraints, each branch of the critical manifold has exactly one eigenvalue which potentially changes sign under the slow flow, namely $\lambda_3$.

\correz{Assume that the fast flow brought the dynamics in the neighbourhood of a stable branch (i.e. the corresponding $\lambda_3<0$) of the critical manifold.} If the branch of the critical manifold which the fast system approaches does not lose stability under the slow flow (the first possible scenario), the system simply exhibits convergence towards the unique stable equilibrium on that branch. In particular, this is always the case when the value $\mathcal{R}_0^{s,i}<1$, which implies monotone convergence of $I \to 0$, maintaining the corresponding eigenvalue negative.

Hence, we are particularly interested in the cases in which $\mathcal{R}_0^{s,i}>1$, and for which the value of the infected population at the endemic equilibrium $I^{*,i}$ potentially makes the corresponding $\lambda_3>0$ (the second possible scenario). We refer to Section \ref{sec:numerics} for numerical simulations of various interesting scenarios.

Since one of the eigenvalues potentially changes its sign during the slow flow, the critical manifold is not uniformly hyperbolic, and thus standard GSPT theory can not be applied. Indeed, one of the original assumptions of classical Fenichel's theory \correz{\cite{fenichel1979geometric,hek2010geometric,jones1995geometric,kuehn2015multiple}} was that of uniform hyperbolicity of the critical manifold, meaning that the eigenvalues corresponding to the fast variables on the critical manifold should be uniformly bounded away from the imaginary axis of the complex plane, to ensure either local stability or instability. However, it is quite common in epidemic modelling that the eigenvalues driving the slow part of the dynamics change sign, necessarily vanishing in the process, and this is actually the mechanisms behind long dormant stages of the epidemic between consequent waves \cite{jardon2021geometric1,jardon2021geometric2,dellamarca2023geometric,kaklamanos2023geometric,achterberg2023minimal}.

Systems of this kind may exhibit a delayed loss of stability; meaning, orbits of the corresponding system remain close to a repelling branch of the critical manifold for a long time, before finally leaving a neighbourhood of said manifold. In order to measure this permanence, a very useful tool is the so-called \textit{entry-exit function} \cite{de2008smoothness,de2016entry}. 

More specifically, in its lowest dimensional formulation, the entry-exit function applies to planar systems of the form 
\begin{equation}\label{eq:entex}
\begin{aligned}
x'&= f(x,y,\varepsilon)x,\\
y'&=\varepsilon g(x,y,\varepsilon),
\end{aligned}
\end{equation}
with $(x,y)\in \mathbb{R}^2$, $g(0,y,0)>0$ and $\textnormal{sign}(f(0,y,0))=\textnormal{sign}(y)$. Note that for $\varepsilon=0$, the $y$-axis consists of normally attracting/repelling equilibria if $y$ is negative/positive, respectively.
\begin{figure}[htbp]\centering
	\begin{tikzpicture}
		\node at (0,0){\includegraphics[scale=0.85]{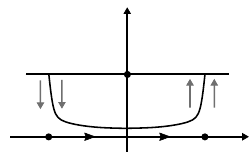}};
		\node at (2,-.9){$y$};
		\node at (0.05,1.15){$x$};
		\node at (-1.1,.2){$x=x_0$};
		\node at (-1,-1.15) {$y_0$};
		\node at ( 1.1,-1.15) {$p_\varepsilon(y_0)$};
	\end{tikzpicture}
	\caption{Visualization of the entry-exit map on the line $x=x_0$. The point $(y,x)=(0,0)$ is non-hyperbolic.} 
	\label{fig:entrex2D}
\end{figure}
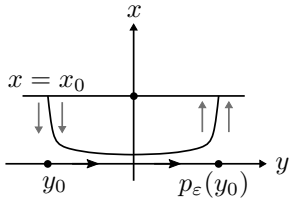

Consider a horizontal line $\{x=x_0\}$, close enough to the $y$-axis to obey the attraction/repulsion assumed above. An orbit of \eqref{eq:entex} that intersects such a line at $y=y_0<0$ (entry) re-intersects it again (exit) at $y=p_\varepsilon(y_0)$, as sketched in Figure~\ref{fig:entrex2D}.

As $\varepsilon \rightarrow 0$, the image of the return map $p_\varepsilon(y_0)$ to the horizontal line $x=x_0$ approaches $p_0(y_0)$ given implicitly by
\begin{equation}\label{eq:pzero}
\int_{y_0}^{p_0(y_0)} \frac{f(0,y,0)}{g(0,y,0)}\textnormal{d}y = 0.
\end{equation}
This construction can be generalized to higher dimensional systems, such as the one we are studying in this paper. For a more precise description of the planar case, we refer to \cite{de2008smoothness,de2016entry} or the preliminaries of \cite{jardon2021geometric1}. For more general theorems, we refer the interested reader to \cite{kaklamanos2022entry,liu2000exchange,neishtadt1987persistence,neishtadt1988persistence,schecter2008exchange}.

This tool has proven to be remarkably useful in particular in epidemic modelling evolving on multiple time scales, as it provides a measure of the dormant phase between subsequent waves of an epidemic \cite{achterberg2023minimal,dellamarca2023geometric,jardon2021geometric1, jardon2021geometric2,kaklamanos2023geometric}. 

Unfortunately, the \correz{classical} entry-exit formulas \correz{fundamentally} rely on two \correz{additional} assumptions, \correz{one of which is not satisfied in our case. The first one is} the monotonicity of the eigenvalue which causes the delayed loss of stability. This is related to the assumption $g(0,y,0)>0$ in the planar case \eqref{eq:entex}. Indeed, as a consequence of our analysis in Section \ref{sec:slow}, when the corresponding slow Basic Reproduction Number $\mathcal{R}_0^{s,i}>1$, the slow system approaches its unique endemic equilibrium through damped oscillations. Since the eigenvalues we are interested in are functions of $I$, they are not monotone. A particularly pathological case is the one in which the interplay between the (mis)information spreading parameters and the epidemic spreading parameters is such that the value of $I$ at the endemic equilibrium (i.e. $I^{*,i}$) coincides with  $\lambda_3=0$. In that case, indeed, the eigenvalue $\lambda_3$ would potentially change sign multiple times throughout the slow flow, rendering analytical results about the permanence of orbits close to a specific branch of the critical manifold impossible to achieve with the currently available theory on entry-exit functions. \correz{We do not explicitly derive the conditions on the parameters for which this scenario happens, as the calculations are remarkably cumbersome and do not contribute to the overall understanding of the dynamics. Indeed, to check whether this is possible or not, one would need to use the endemic equilibrium values of $I^{*,i}$, namely \eqref{EE_C00}, \eqref{EE_C01} and \eqref{EE_C02}, and impose that they lie in the unstable region of the corresponding branch of the critical manifold; recall \eqref{eq:crit_manif}.} 

The second \correz{classical assumption of} the entry-exit function is the separation of eigenvalues. Namely, under the slow flow, the eigenvalues which provides (delayed) loss of stability should never intersect the remaining ones. A first result towards a generalization in this direction was recently obtained in \cite{kaklamanos2022entry}. However, the computations to check separation in our case are \correz{quite involved}, and considering that \correz{another hypothesis is} not satisfied, we do not include them.

Instead, we conjecture that, away from pathological scenarios in which the damped convergence towards the endemic equilibrium corresponds to multiple sign changes of the eigenvalue we focus on, similar entry-exit relations are to be expected, and we showcase some concordant simulations in the next section.

\section{Numerical simulations}\label{sec:numerics}

In this section, we provide various simulations of system \eqref{mod_3}, highlighting noteworthy transient and asymptotic dynamics depending on the selected values of the parameters in the system. Recall \eqref{eq:crit_manif} for the three branches of the critical manifold.

The parameters used in each simulation are specified in the corresponding caption. We remark that our choice of the parameters was dictated by the conditions needed to clearly visualize each scenario, rather than an attempt to closely match to real-world values of each quantity. The estimation of such parameters is outside the scope of the present work; moreover, it is extremely case-sensitive, and for the information layer, quite difficult to deduce. Hence, we renounce realism in favour of an increased clarity of the figures.

In Figure \ref{fig:case_1}, the system exhibits two entry-exit phenomena. First, the orbit approaches $\mathcal{C}_0^1$ (representing an absence of skeptical individuals in the population). After the peak of infection, the orbit ``jumps'' from $\mathcal{C}_0^1$ to $\mathcal{C}_0^2$ (representing an absence of misinformed individuals in the population). Around $t\approx 150$, the orbit ``jumps'' again, eventually settling on $\mathcal{C}_0^1$, with a corresponding high asymptotic value of infected individuals in the population. In Figure \ref{fig:case_6} instead, the orbit approaches first the branch $\mathcal{C}_0^1$ and later settles on $\mathcal{C}_0^2$.

In Figure \ref{fig:case_2}, we observe again the same two entry-exit phenomena, although the slow passage close to $M=0$ is much longer than the one in Figure \ref{fig:case_1}. Recall Section \ref{sec:threshold}, in which we identified a line of equilibria for the fast system. In the full system, only one equilibrium is present on this line, and we provide its formulation in Appendix \ref{app_eq}. The system then converges towards this equilibrium, leaving the critical manifold without ever approaching any of its branches again.

The next two numerical simulations showcase how the (mis)information spread may negatively or positively affect the evolution of the epidemic, depending on whether the slow dynamics evolves on the second or on the third branch of the critical manifold; recall \eqref{eq:crit_manif}.

In Figure \ref{fig:case_3}, we illustrate an example of how \correz{the presence of misinformed individuals and the absence of skeptical individuals} can negatively affect the evolution of the epidemic, making an infectious disease which would naturally die out become endemic; we refer to Section \ref{sec:lento_1} for more context.

In Figure \ref{fig:case_4}, instead, we illustrate an example of how \correz{the presence of skeptical individuals and the absence of misinformed individuals} can positively affect the evolution of the epidemic, making an infectious disease which would naturally be endemic evolve towards extinction; we refer to Section \ref{sec:lento_2} for more context.

Finally, as illustrated in Figure \ref{fig:case_5}, during our numerical exploration we found numerous parameter sets for which the system ``jumps'' between $\mathcal{C}_0^1$ and $\mathcal{C}_0^2$ multiple times before finally converging to an equilibrium. We remark that our theoretical analysis of system \eqref{mod_3} did \emph{not} exclude the possibility of a GSPT-specific class of periodic orbits, namely \emph{Mixed Mode Oscillations} (MMOs) \cite{brons2008introduction}. Indeed, the global stability analysis of the fast subsystem \eqref{mod_1} was performed assuming $I$ fixed. \correz{A situation in which an orbit keeps ``jumping'' between $\mathcal{C}_0^1$ and $\mathcal{C}_0^2$, for example, might exist. In fact,} once near any of the branches, if the corresponding $\mathcal{R}_0^{s,i}>1$, the solution will approach the EE of the slow flow. If such EE is in the repelling part of the critical manifold, the orbit would eventually leave its vicinity (we refer to Section \ref{sec:entry-exit} for a deeper explanation of this phenomenon), only to approach the other branch, and this cycle would repeat forever. However, we were not able to find parameter values for which we achieve this scenario in system \eqref{mod_3}. 

\begin{figure}[h]
     \centering
     \begin{subfigure}[b]{0.49\textwidth}
         \centering
\includegraphics[width=0.9\textwidth]{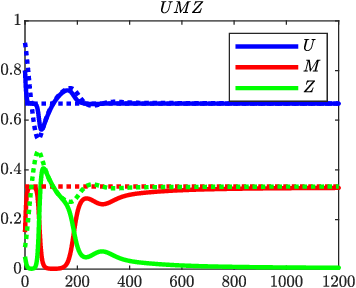}
         \caption{}
     \end{subfigure}
     \hfill
     \begin{subfigure}[b]{0.49\textwidth}
         \centering
\includegraphics[width=0.9\textwidth]{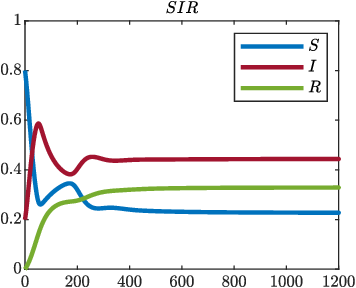}
         \caption{}
     \end{subfigure}
     \caption{Values of the parameters: $b_1 = 1.5$, $b_2=0.9$, $\beta = 6$, $\varepsilon = 1/100$, $\eta = 0.08$, $\correz{\mu_1}=\correz{\mu_2} = 1$, $\gamma = 0.8$, $K=0.9$. Initial conditions: $(U,M,Z,S,I,R) = (0.9, 0.15, 0.05, 0.8, 0.2, 0)$. Solid lines: corresponding orbit of system \eqref{mod_3}; dashed lines: branches of the critical manifold \eqref{eq:crit_manif}, except for $M=0$ and $Z=0$ which are not shown for clarity. We observe two entry-exit phenomena, with the solution approaching $\mathcal{C}_0^1$, then $\mathcal{C}_0^2$, before the asymptotic convergence towards the EE on $\mathcal{C}_0^1$.}
     \label{fig:case_1}
\end{figure}

\begin{figure}[h]
     \centering
     \begin{subfigure}[b]{0.49\textwidth}
         \centering
\includegraphics[width=0.9\textwidth]{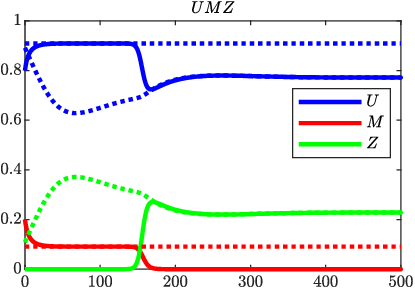}
         \caption{}
     \end{subfigure}
     \hfill
     \begin{subfigure}[b]{0.49\textwidth}
         \centering
\includegraphics[width=0.9\textwidth]{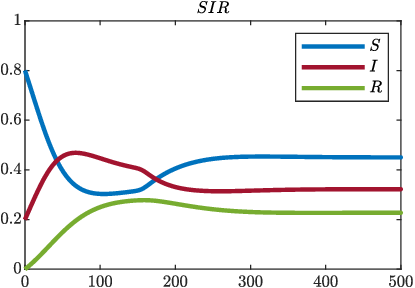}
         \caption{}
     \end{subfigure}
     \caption{Values of the parameters: $b_1 = 1.1$, $b_2=0.92$, $\beta = 6$, $\varepsilon = 1/100$, $\eta = 0.7$, $\correz{\mu_1}=\correz{\mu_2} = 1$, $\gamma = 1.2$, $K=0.9$. Initial conditions: $(U,M,Z,S,I,R) = (0.8, 0.2, 0, 0.8, 0.2, 0)$. Solid lines: corresponding orbit of system \eqref{mod_3}; dashed lines: branches of the critical manifold \eqref{eq:crit_manif}, except for $M=0$ and $Z=0$ which are not shown for clarity. We observe one entry-exit phenomenon, with the solution approaching $\mathcal{C}_0^1$ before the asymptotic convergence towards the EE on $\mathcal{C}_0^2$.}
     \label{fig:case_6}
\end{figure}

\begin{figure}[h!]
     \centering
     \begin{subfigure}[b]{0.49\textwidth}
         \centering
\includegraphics[width=0.9\textwidth]{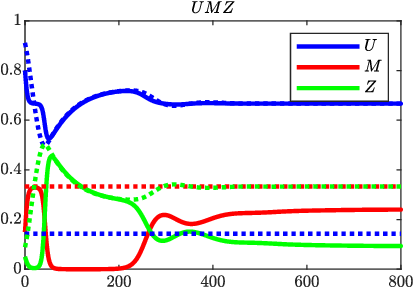}
         \caption{}
     \end{subfigure}
     \hfill
     \begin{subfigure}[b]{0.49\textwidth}
         \centering
\includegraphics[width=0.9\textwidth]{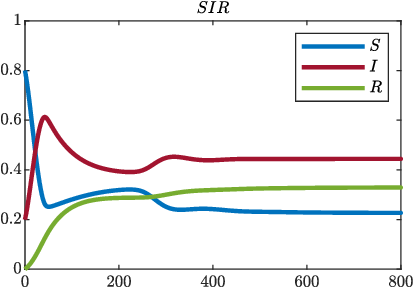}
         \caption{}
     \end{subfigure}
     \caption{Values of the parameters: $b_1 = 1.5$, $b_2=0.9$, $\beta = 7$, $\varepsilon = 1/100$, $\eta = 0.08$, $\correz{\mu_1}=\correz{\mu_2} = 1$, $\gamma = 0.8$, $K=0.9$. Initial conditions: $(U,M,Z,S,I,R) = (0.9, 0.15, 0.05, 0.8, 0.2, 0)$. Solid lines: corresponding orbit of system \eqref{mod_3}; dashed lines: branches of the critical manifold \eqref{eq:crit_manif}, except for $M=0$ and $Z=0$ which are not shown for clarity. We observe again two entry-exit phenomena, first close to $\mathcal{C}_0^2$, then close to $\mathcal{C}_0^1$, before the system converges towards the co-existence equilibrium.}
     \label{fig:case_2}
\end{figure}

\begin{figure}[h!]
     \centering
     \begin{subfigure}[b]{0.49\textwidth}
         \centering
\includegraphics[width=0.9\textwidth]{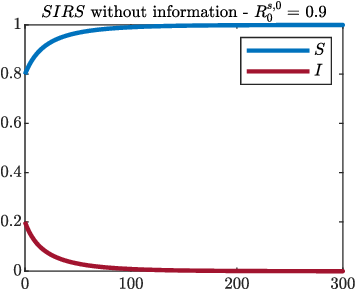}
         \caption{}
     \end{subfigure}
     \hfill
     \begin{subfigure}[b]{0.49\textwidth}
         \centering
\includegraphics[width=0.9\textwidth]{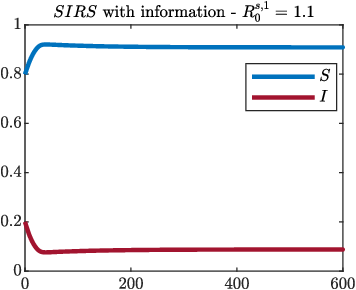}
         \caption{}
     \end{subfigure}
     \caption{Comparison between (a) the SIRS system \eqref{lento_0} without information (i.e., evolving on $\mathcal{C}_0^0$), naturally converging to the corresponding DFE, and (b) the SIRS system \eqref{lento_1} evolving on the misinformation branch of the critical manifold (i.e., $\mathcal{C}_0^1$), converging to the corresponding EE. Values of the parameters: $b_1 = 6.7041$, $b_2=5.9041$, $\beta = 4.8429$, $\varepsilon = 1/50$, $\eta = 0.1$, $\correz{\mu_1}=\correz{\mu_2} = 5.2143$, $\gamma = 0.1667$, $K=0.9$. Initial conditions: $(S,I,R) = (0.8, 0.2, 0)$. We do not show the evolution in time of $R$, since it remains close to 0 for all times.}
     \label{fig:case_3}
\end{figure}

\begin{figure}[h!]
     \centering
     \begin{subfigure}[b]{0.49\textwidth}
         \centering
\includegraphics[width=0.9\textwidth]{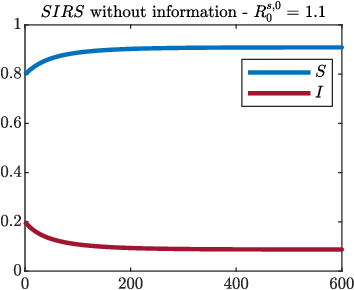}
         \caption{}
     \end{subfigure}
     \hfill
     \begin{subfigure}[b]{0.49\textwidth}
         \centering
\includegraphics[width=0.9\textwidth]{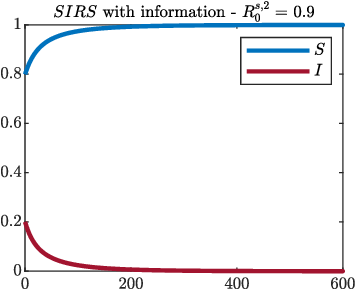}
         \caption{}
     \end{subfigure}
     \caption{Comparison between (a) the SIRS system \eqref{lento_0} without information (i.e., evolving on $\mathcal{C}_0^0$), naturally converging to the corresponding EE, and (b) the SIRS system \eqref{lento_2} evolving on the skeptical branch of the critical manifold (i.e., $\mathcal{C}_0^2$), converging to the corresponding DFE. Values of the parameters: $b_1 = 5.9041$, $b_2=6.7041$, $\beta = 5.919$, $\varepsilon = 1/50$, $\eta = 0.1$, $\correz{\mu_1}=\correz{\mu_2} = 5.2143$, $\gamma = 0.1667$, $K=0.9$. Initial conditions: $(S,I,R) = (0.8, 0.2, 0)$. We do not show the evolution in time of $R$, since it remains close to 0 for all times.}
     \label{fig:case_4}
\end{figure}

\begin{figure}[h!]
     \centering
     \begin{subfigure}[b]{0.49\textwidth}
         \centering
\includegraphics[width=0.9\textwidth]{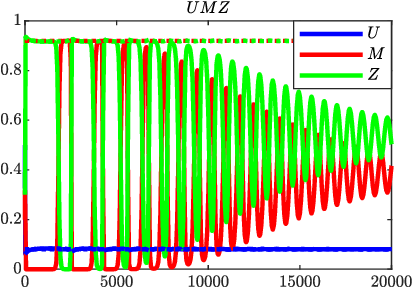}
         \caption{}
     \end{subfigure}
     \hfill
     \begin{subfigure}[b]{0.49\textwidth}
         \centering
\includegraphics[width=0.9\textwidth]{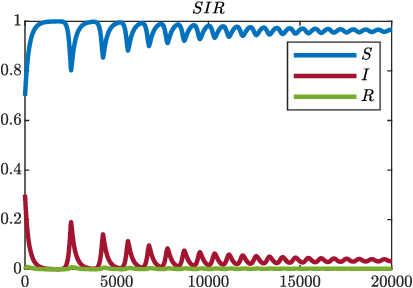}
         \caption{}
     \end{subfigure}
     \caption{Values of the parameters: $b_1 = 10.79$, $b_2=10.44$, $\beta =  1.0897$, $\varepsilon = 1/100$, $\eta = 0.4570$, $\correz{\mu_1}=\correz{\mu_2} = 0.87$, $\gamma = 0.0776$, $K=0.9$. Initial conditions: $(U,M,Z,S,I,R) = (0.5,0.2,0.3,0.7,0.3,0)$. Solid lines: corresponding orbit of system \eqref{mod_3}; dashed lines: branches of the critical manifold \eqref{eq:crit_manif}, except for $M=0$ and $Z=0$ which are not shown for clarity. We observe several entry-exit instances, before the system converges towards the co-existence equilibrium. We truncated the integration at $t=20000$ for clarity.}
     \label{fig:case_5}
\end{figure}

\section{\correz{Discussion and conclusions}}\label{sec:concl}
In this paper, we introduced a six dimensional nonlinear info-epidemic coupled model, in which the first three equations describe the fast dynamics (infodemic layer) while the last three the slow one (epidemic layer). First, we analyzed the infodemic model, which describes the dynamics in time of the Unaware, Misinformed and Skeptical populations. The reduced system has four different equilibrium points feasible and stable under certain conditions. Moreover, we proved that the system can exhibit three different transcritical bifurcations between the Misinformed-Skeptical-Free equilibrium and the remaining three ones (Skeptical-Free, Misinformed-Free and coexistence equilibria). 

The set of the first three equilibria of the fast model plays a crucial role in the dynamics of the coupled info-epidemic model, since these equilibria provide the formulation of the three branches of the critical manifold of the complete system. \correz{Subsequently, we studied the behaviour of the slow system on each of the three branches of the critical manifold. The first branch is the Misinformed-Skeptical-Free one, while the second and the third ones are the Skeptical-Free and the Misinformed-Free equilibria, respectively. Interestingly, we noticed that the slow system behaves as a classical SIRS model on the first branch of the critical manifold, corresponding to the Misinformed-Skeptical-Free equilibrium. Secondly, for the second branch of the critical manifold, corresponding to the Skeptical-Free equilibrium point, we found that the infection rate of the obtained slow SIRS model is higher that the infection rate of a classical SIRS model; this confirm the fact that the absence of skeptical individuals increases the spread of the disease. Since the two slow SIRS models have a basic reproduction number $\mathcal{R}_0^{s,0}$ and $\mathcal{R}_0^{s,1}$, respectively, which determines when the system converges to the Disease Free Equilibrium or there exist only a unique Endemic Equilibrium, as for a simple SIRS model. Assuming that the first one is smaller than $1$ and the second one bigger than $1$ means that the disease persists (thus becoming endemic), rather than dying out, due to the presence of misinformed individuals. Lastly, for the third branch of the critical manifold, related to the Misinformed-Free equilibrium, the obtained slow SIRS model has an infection rate smaller than the infection rate of the simple SIRS model, meaning that the absence of misinformed individuals and the presence of skeptical one decrease the disease infection rate. Moreover, the infection rate in this case directly depends on the number of the infected individuals. Diametrically opposite conclusions, compared to the second case, can be obtained regarding the basic reproduction numbers of the second and the third slow SIRs models respectively, i.e. the presence of enough skeptical individuals to the disease-related misinformation might lead to the extinction of the disease.}

\correz{Our analytical results are illustrated by the} numerical simulations, which showcase various possible asymptotic behaviours of the system, with orbits approaching different equilibria as $t\to +\infty$, with remarkable transient behaviours. \correz{Moreover}, we observe multiple instances of delayed loss of stability, for systems which are not entirely covered by theoretical results on the matter.


The analysis we carried out in this work gave rise to various questions. We elaborate on three of them here, considering their potential for future research projects.

Firstly, is it possible to obtain results similar to the ones contained in this paper, however for an unknown function $\bar{b_2}$, assumed only to be decreasing in $I$? This function would still be ``frozen'' in the fast flow, \correz{since $I$ is one of the slow variables of system \eqref{mod_3},} so the additional difficulties would only impact the slow flow. However, for SIR models, there are already known results in literature with minimal assumptions on the incidence, recall \cite{lahrouz2012complete}.

\correz{Secondly, would our model exhibit interesting new dynamics giving us new biological insight about the epidemic and the infodemic, if we complicated the information layer of the dynamics?} For example, if we allowed misinformed individuals to become skeptical, and viceversa? Moreover, we speculate that a different choice for the function $\correz{\Bar{\beta}(M,Z)}$ might result in MMOs for a system of ODEs very similar to \eqref{mod_3}. However, changing $\correz{\Bar{\beta}(M,Z)}$ would require to re-do most of the analysis we carried out, at the very least on the slow systems. Nevertheless, we believe that this is an extremely promising outlook for a future work on similar compartmental models.

Thirdly, on a more theoretical note, we pose the following challenging question: is it possible to generalize the entry-exit function to systems with non-monotone slow dynamics? A minimal model to test this could be a 2D slow-fast system, with the slow variable exhibiting (damped) oscillations on the critical manifold. Consider, for example, $\Dot{x}=\sin x$, or  $\Dot{x}=\sin x/(x^2+1)$, on the corresponding critical manifold. However, such a question is far from the scope of this paper, and we leave it open for a more theoretical work in the future.\\

\textbf{Acknowledgments} 
This work is the output of the collaboration started during the Workshop MSE: Modellistica Socio-Epidemiologica (Socio-epidemiological modelling), held in Naples in May 2023. 

I.M.B. thanks Piergiorgio Castioni and Irene Ferri for the useful discussions on the problem at the  Winter Workshop on Complex Systems 2022 (WWCS2022). 

I.M.B. has been supported by Fondazione di Sardegna, project 2022/2023: CUP-J83C21000060007, by ``Finanziamento GNCS Giovani Ricercatori 2021-2022”, by INdAM – GNCS Project 2023: CUP-E53C22001930001 and by the project TIGRECO funded by the MUR Progetti di Ricerca di Rilevante Interesse Nazionale (PRIN) Bando 2022, Grant 20227TRY8H, CUP-J53D23003630001. M.S. and S.S. were supported by the Italian Ministry for University and Research (MUR) through the PRIN 2020 project ``Integrated Mathematical Approaches to Socio-Epidemiological Dynamics'' (No. 2020JLWP23). S.S. is member of the ``Gruppo Nazionale per l'Analisi Matematica e le sue Applicazioni" (GNAMPA) of the ``Istituto Nazionale di Alta Matematica" (INdAM).

\appendix \section{The equilibrium analysis of the complete system}\label{app_eq}
One of the strengths of GSPT is the ability to iteratively reduce the dimensionality of the system under study. In this paper, we started with a 6D system, however we always considered at most 3 of those dimensions at a time, with a great advantage in terms of computational burden. Nevertheless, it might be instructive to showcase what one can deduce on the local stability of the equilibria of the full 6D model.

We now list the equilibria of model \eqref{mod_3}, which are exactly the equilibria of the slow systems on the corresponding branches of the critical manifold. For ease of notation, we maintain the order of the variables of system \eqref{mod_3} and indicate with $*$ the position where the value is different from 0 and 1. 

We hence obtain the analytical expression of seven equilibrium points, namely 
\begin{itemize}
\item[(i)] Misinformation-disease free equilibrium (MDFE) $E_0=(1,0,0,1,0,0)$, which is always feasible. This corresponds to the Disease Free Equilibrium on the first branch of the critical manifold $\mathcal{C}_0^0$ \eqref{eq:crit_manif}.
\item[(ii)] Misinformation free equilibrium (MFE) $E_1=(1,0,0,*,*,*)$, with 
$$S^{*,1} = \dfrac{\gamma+\correz{\mu_2} }{\beta}, \quad  I^{*,1} = \dfrac{(\beta - \gamma - \correz{\mu_2})(\eta + \correz{\mu_2})}{\beta (\gamma + \eta + \correz{\mu_2})}, \qquad R^{*,1} = \dfrac{\gamma I^{*,1}}{\eta + \correz{\mu_2}},$$ 

which is feasible if $\beta > \gamma + \correz{\mu_2}$, meaning if $\mathcal{R}_0^{s,0}>1$ (recall \eqref{eq:rzerozero}). This corresponds to the Endemic Equilibrium on the first branch of the critical manifold $\mathcal{C}_0^0$ \eqref{eq:crit_manif}.
\item[(iii)] Disease-skeptical individuals free equilibrium (DSIFE) $E_2=(*,*,0,1,0,0)$, with 
$$ U^{*,2} = \dfrac{\correz{\mu_1}}{b_1} ,\quad M^{*,2} = \dfrac{b_1- \correz{\mu_1}}{b_1}, $$
which is feasible if $ \correz{\mu_1}<b_1$. This corresponds to the Disease Free Equilibrium on the second branch of the critical manifold $\mathcal{C}_0^1$ \eqref{eq:crit_manif}.

\item[(iv)] Skeptical individuals free equilibrium (SIFE) $E_3=(*,*,0,*,*,*)$, with 


$$ U^{*,3} = \dfrac{\correz{\mu_1}}{b_1}, \quad M^{*,3} = \dfrac{b_1- \correz{\mu_1}}{b_1}, \quad S^{*,3}= \dfrac{\gamma + \correz{\mu_2}}{\bar{\beta}}, \qquad I^{*,3} = \dfrac{(\bar{\beta}-\gamma - \correz{\mu_2} )(\eta + \correz{\mu_2})}{\bar{\beta} (\gamma + \eta + \correz{\mu_2})}, \qquad R^{*,3} = \dfrac{\gamma I^{*,3}}{\eta + \correz{\mu_2}}$$ with  $$ \bar{\beta} = \beta\dfrac{2b_1-\correz{\mu_1}}{b_1},$$ 
which is feasible if $ \correz{\mu_1} < b_1$ and $\beta > \dfrac{b_1(\gamma + \correz{\mu_2})}{2 b_1-\correz{\mu_1}}$, meaning if $\mathcal{R}_0^{s,1}>1$ (recall \eqref{eq:rzerouno}). This corresponds to the Endemic Equilibrium on the second branch of the critical manifold $\mathcal{C}_0^1$ \eqref{eq:crit_manif}.

\item[(v)] Disease-misinformed individuals free equilibrium (DMIFE) $E_4=(*,0,*,1,0,0)$, with 

$$ U^{*,4} = \dfrac{\correz{\mu_1}}{b_2}, \quad Z^{*,4} = \dfrac{b_2- \correz{\mu_1}}{b_2}, $$
which is feasible if $ \correz{\mu_1} < b_2$. This corresponds to the Disease Free Equilibrium on the third branch of the critical manifold $\mathcal{C}_0^2$ \eqref{eq:crit_manif}.

\item[(vi)] Misinformed individual free equilibrium (MIFE)
$E_5=(*,0,*,*,*,*)$, with 
$$ U^{*,5} = \dfrac{\correz{\mu_1} (1-KI^{*,5})}{b_2} ,\quad Z^{*,5} = \dfrac{b_2- \correz{\mu_1}(1-KI^{*,5})}{b_2} ,\quad 
S^{*,5} = \dfrac{(\gamma + \correz{\mu_2}) ( 2b_2 - \correz{\mu_1}(1-KI^{*,5})}{\beta b_2},$$
$$I^{*,5} = \dfrac{(\gamma+\correz{\mu_2})(\beta b_2 - (\gamma+\correz{\mu_2})(2b_2-\correz{\mu_1}))}{\beta b_2 (\eta + \correz{\mu_2} + \gamma) +\correz{\mu_1} K ( \eta+\correz{\mu_2})(\gamma+\correz{\mu_2})}, \quad R^{*,5} = \dfrac{\gamma I^{*,5}}{\eta + \correz{\mu_2}}, $$

which is feasible if $\mathcal{R}_0^{s,2}>1$ (recall \eqref{eq:rzerodue}). This corresponds to the Endemic Equilibrium on the third branch of the critical manifold $\mathcal{C}_0^2$ \eqref{eq:crit_manif}.

\item[(vii)] Coexistence equilibrium $E_*=(*,*,*,*,*,*)$, with 
$$ U^* = \dfrac{\correz{\mu_1}}{b_1}, \quad M^* =\dfrac{(\gamma + \correz{\mu_2})(1+Z^*)- \beta S^*}{\beta S^*} , \quad  Z^* = \dfrac{  b_1(\gamma+ \correz{\mu_2})+ \beta S^* }{b_1  (\gamma + \correz{\mu_2} - \beta S^*)},   
\quad S^* = \dfrac{ (\eta +\correz{\mu_2})- (\eta +\correz{\mu_2}+ \gamma) I^* }{(\eta +\correz{\mu_2})}, \quad $$ 
$$I^* = \dfrac{b_1-b_2}{K b_1},\quad R = \dfrac{\gamma I^*}{\eta + \correz{\mu_2}} , $$
which is feasible if $b_1>b_2$, $(\eta+\correz{\mu_2})(1-I^*)>\gamma I^*$, $\beta S^* < \gamma +\correz{\mu_2}$ and $M^*>0$.

This equilibrium does not lie on any branch of the critical manifold, and as such it was not identified through our analysis in the main body of the paper; however it does lie on the line of equilibria described in Section \ref{sec:threshold}. 
\end{itemize}
The feasibility conditions were analyzed in the main body of this work; in what follows, we will focus our attention on the local stability of the equilibrium points, assuming that they are indeed feasible. For ease of notation, every time $U^*$, $M^*$, $Z^*$, $S^*$, $I^*$ and $R^*$ are mentioned, we refer to the expressions they assume at the corresponding equilibrium. 

\begin{itemize}
\item[(i)] The misinformation-disease free equilibrium, $E_0$, it is stable iff $\correz{\mu_1} > \max{ \left\lbrace b_1,b_2 \right\rbrace}$ holds. In fact the eigenvalues are $\lambda_1=-\correz{\mu_1}$, $\lambda_2=-\varepsilon \correz{\mu_2}$, $\lambda_3= b_1-\correz{\mu_1}$, $\lambda_4= b_2-\correz{\mu_1}$, $\lambda_5= -\varepsilon (\eta + \correz{\mu_2})$, $\lambda_6= -\varepsilon (\beta + \gamma+ \correz{\mu_2})$. 

\item[(ii)] The eigenvalues corresponding to the misinformation free equilibrium $E_1$  are $\lambda_1=-\correz{\mu_1}$, $\lambda_2=-\varepsilon \correz{\mu_2}$,  $\lambda_3= b_1-\correz{\mu_1}$,  $\lambda_4= - \correz{\mu_1} + \dfrac{b_2}{1-KI^*}$, while $\lambda_5$ and  $\lambda_6$ are the roots of the second degree polynomial in $\lambda$
$$ \lambda^2 +\lambda (\beta p_1+1)(\correz{\mu_2}+\eta) \varepsilon+p_1 \beta (\correz{\mu_2}+\eta)(\correz{\mu_2}+\gamma+\eta) \varepsilon^2=0. $$
Since the coefficients of the second degree polynomial introduced above are all positive, notice that if the roots are complex we have that the real part of both $\lambda_5$ and  $\lambda_6$ are negative, while if the roots are real for the Descartes' rule of signs they are negative. Therefor the equilibrium $E_1$ it is stable iff $\correz{\mu_1} > \max \left\lbrace b_1, \dfrac{b_2}{1-KI^*}  \right\rbrace$ (recall that the condition $ \correz{\mu_1} < b_1$ is needed for the feasibility of $E_3$). 

\item[(iii)] The disease-skeptical individuals free equilibrium $E_2$ it is stable iff $b_1-b_2 >0$ and $\beta < \dfrac{b_1(\gamma + \correz{\mu_2})}{2 b_1-\correz{\mu_1}}$. The corresponding eigenvalues are  $ \lambda_1 =-\varepsilon \correz{\mu_2}$, $ \lambda_2=-\varepsilon(\eta +\correz{\mu_2})$, $\lambda_3= - \dfrac{\correz{\mu_1}  (b_1-b_2)}{b_1}$, $\lambda_4= -\correz{\mu_1}$, $\lambda_5= \correz{\mu_1} - b_1$ and $\lambda_6= -\varepsilon \dfrac{b_1( \gamma+ \correz{\mu_2})-\beta(2b_1 -\correz{\mu_1}) }{b_1}$. 

\item[(iv)] The eigenvalues that corresponds to the skeptical individuals free equilibrium $E_3$ are $\lambda_1 = -\varepsilon\correz{\mu_2}$, $\lambda_2=\correz{\mu_1} - b_1 $, $\lambda_3 = -\correz{\mu_1}$, 
$$\lambda_4 = -\dfrac{\correz{\mu_1}\left[ b_1^2 (\eta +\correz{\mu_2})(\gamma+\correz{\mu_2}) - \beta(2 b_1 - \correz{\mu_1}) p_2 \right]}{b_1 \left[ p_2+ \beta \gamma (2 b_1- \correz{\mu_1})+ b_1(\gamma +\correz{\mu_2}^2)\right]}$$
where $p_2=b_2(\eta +\gamma +\correz{\mu_2}) - \gamma b_1$, while $\lambda_5$ and $\lambda_6$ are the roots of the following second degree polynomial in $\lambda$:
$$\lambda^2+ \dfrac{\varepsilon(\eta + \correz{\mu_2}) \left[ \eta b_1+(2 b_1-\correz{\mu_1}) \beta \right]}{b_1(\gamma + \eta + \correz{\mu_2})} \lambda+ \frac{\varepsilon^2(\eta + \correz{\mu_2}) \left[ \beta(2 b_1-\correz{\mu_1})-b_1(\correz{\mu_2} + \gamma) \right] }{b_1} =0. $$
Notice that for the Routh-Hurwitz stability criterion for a system of second order, since all the coefficients of $p(\lambda)$ are all positive its roots have negative real parts. Thus for the stability of $E_5$, 
$$ \left[ b_1^2 (\eta +\correz{\mu_2})(\gamma+\correz{\mu_2}) - \beta(2 b_1 - \correz{\mu_1}) p_2 \right]>0$$ must hold.

\item[(v)] The disease-misinformed individuals free equilibrium $E_4$ it is stable iff $b_1 <b_2$ and $\beta < \dfrac{(\gamma +\correz{\mu_2})( 2 b_2- \correz{\mu_1})}{b_2}$. The corresponding eigenvalues are  $ \lambda_1 =-\varepsilon \correz{\mu_2}$, $ \lambda_2=-\varepsilon(\eta +\correz{\mu_2})$, $\lambda_3= - \dfrac{\correz{\mu_1} (b_2-b_1)}{b_2}$, $\lambda_4= -\correz{\mu_1}$, $\lambda_5= -b_2+ \correz{\mu_1}$ and $\lambda_6= -\varepsilon \dfrac{(2b_2 -\correz{\mu_1})(\gamma +\correz{\mu_2})-\beta b_2}{2b_2 -\correz{\mu_1}}$. 

\item[(vi)] The eigenvalues of the misinformed individual free equilibrium $E_5$ are $\lambda_1 = -\varepsilon\correz{\mu_2}$, $\lambda_2=-\correz{\mu_1} + b_1 U^*$, $\lambda_3 = -\correz{\mu_1}$, while the remaining three eigenvalues are the roots of a third degree polynomial, to involved to be reported here. Numerical evidence, and our analysis in the main body of this work, show that there exist sets of parameters for which $E_5$ it is stable.
 

\item[(vii)] Computing the eigenvalues of the coexistence equilibrium $E_*$ we obtain $\lambda_1 = -\varepsilon\correz{\mu_2}$, $\lambda_2 = -\correz{\mu_1}$ while the remaining four ones are the roots of a fourth degree polynomial in $\lambda$, which we chose not to include here, since it doesn't provide useful information. We refer to Section \ref{sec:numerics} for a choice of parameters for which the coexistence equilibrium is asymptotically (hence, locally) stable.

 \end{itemize}

\section{Global Stability}\label{app_gs}

\globalMSFE*

\begin{proof} 
Summing the equations \eqref{info_eqs}, and recalling that $U \leq 1$, we obtain
    \begin{align*}
        \dfrac{\text{d}}{\text{d}t}(M+Z) =& \left(b_1 M + \dfrac{b_2}{1-KI}Z \right)U - \correz{\mu_1} (M+Z)\\[2pt]
        \leq & \left(b_1 M + \dfrac{b_2}{1-KI}Z \right) - \correz{\mu_1} (M+Z) \\
        \leq & \left( \max\left\lbrace b_1,  \dfrac{b_2}{1-KI}\right\rbrace - \correz{\mu_1} \right) (M+Z)\\
        = & \correz{\mu_1} \left(  \R0^{\text{f}} -1\right) (M+Z).
    \end{align*}
    If $\R0^{\text{f}}<1$, then it follows that
    \begin{equation*}
        \lim_{t \to + \infty} (M(t) + Z(t)) = 0 \text{ exponentially},
    \end{equation*}
and on the set $\{M = Z = 0\}$, $U \to 1$. We can conclude that the equilibrium $\tilde{E}_0$ is globally exponentially stable if $\R0^{\text{f}}<1$.
\end{proof}

\globalSFE*

\begin{proof}
    It is easy to see that, when $\dfrac{b_2}{\correz{\mu_1} (1-KI)} < 1$, $Z \to 0$ exponentially, with the same argument used in  Proposition \ref{global_MSFE}. Hence, we assume $Z=0$, and we prove global asymptotic stability of the SFE through the use of the classic Goh-Lotka-Volterra Lyapunov function \cite{cangiotti2023survey}:
    $$
    V(U,M)\coloneqq U-U^* \ln U + M -M^* \ln M,
    $$
    where we introduce $U^*$ and $M^*$ as the corresponding equilibrium values, for ease of notation. Then,
    \begin{align*}
        \dfrac{\text{d}}{\text{d}t} V(U,M) &= \dfrac{\text{d}}{\text{d}t} U -\dfrac{U^*}{U} \left( \dfrac{\text{d}}{\text{d}t}U \right)+ \dfrac{\text{d}}{\text{d}t} M -\dfrac{M^*}{M} \left( \dfrac{\text{d}}{\text{d}t}M \right)\\
        &=\correz{\mu_1} -b_1 UM-\correz{\mu_1} U -U^* \left(\dfrac{\correz{\mu_1}}{U} - b_1 M -\correz{\mu_1}\right)+M(b_1 U-\correz{\mu_1})-M^*(b_1 U-\correz{\mu_1})\\
        &= (U-U^*)\left(\dfrac{\correz{\mu_1}}{U} - b_1 M -\correz{\mu_1}\right)+b_1(U-U^*)(M-M^*)\\
        &= (U-U^*)\left(\dfrac{\correz{\mu_1}}{U}-\correz{\mu_1} - b_1 M^*\right)\\
        &=(U-U^*)\left(\dfrac{\correz{\mu_1}}{U}-\correz{\mu_1} - b_1 (1-U^*)\right).
    \end{align*}
Using the fact that $\correz{\mu_1} = b_1 U^*$, we obtain
    \begin{align*}
        \dfrac{\text{d}}{\text{d}t} V(U,M) &= \dfrac{(U-U^*)}{U}(b_1 U^* - b_1 U U^* - b_1 U + b_1 U U^*)\\
        &= -\dfrac{b_1}{U}(U-U^*)^2<0,
    \end{align*}
    which concludes the proof.
\end{proof}

\globalMFE*
\begin{proof}
    It is easy to see that, when $\dfrac{b_1}{\correz{\mu_1}} <1$, $M \to 0$ exponentially, with the same argument used in  Proposition \ref{global_MSFE}. Hence, we assume $M=0$, and we prove global asymptotic stability of the SFE through the use of the classic Goh-Lotka-Volterra Lyapunov function \cite{cangiotti2023survey}:
    $$
    V(U,Z)\coloneqq U-U^* \ln U + Z -Z^* \ln Z,
    $$
    where we introduce $U^*$ and $Z^*$ as the corresponding equilibrium values, for ease of notation. The result follows from computations which are almost identical to the ones in Proposition \ref{global_SFE}; hence, we do not repeat them.
\end{proof}

\globalSFEb*

\begin{proof}
    To prove the global asymptotic stability of the SFE we use of the following Lyapunov function:
    \begin{equation*}
        V(U,M,Z) \coloneqq U-U^* \ln U + M - M^* \ln M + Z,
    \end{equation*}
    where $U^*$ and $M^*$ are the corresponding equilibrium values, as in proposition \ref{global_SFE}. Then
    \begin{align*}
        \dfrac{\text{d}}{\text{d}t} V(U,M) &= \dfrac{\text{d}}{\text{d}t} U -\dfrac{U^*}{U} \left( \dfrac{\text{d}}{\text{d}t}U \right)+ \dfrac{\text{d}}{\text{d}t} M -\dfrac{M^*}{M} \left( \dfrac{\text{d}}{\text{d}t}M \right) + \dfrac{\text{d}}{\text{d}t} Z \\ 
        & = (U-U^*) \left( \dfrac{\correz{\mu_1}}{U} - b_1 M - \dfrac{b_1}{1-KI} Z -\correz{\mu_1} \right) +(M-M^*) (b_1 U - \correz{\mu_1}) + Z \left(\dfrac{b_2}{1-KI} U -\correz{\mu_1}\right) \\
        & = (U-U^*) \left( \dfrac{\correz{\mu_1}}{U} -\correz{\mu_1}\right) + U^* \left( b_1 M + \dfrac{b_2}{1-KI}Z \right) - \correz{\mu_1}(M+Z) - M^* (b_1 U -\correz{\mu_1}).
    \end{align*}
Since $\dfrac{b_2}{1-KI}<b_1$ we obtain
\begin{align*}
        \dfrac{\text{d}}{\text{d}t} V(U,M) &\le (U-U^*) \left( \dfrac{\correz{\mu_1}}{U} -\correz{\mu_1}\right) + b_1 U^* \left( M + Z \right) - \correz{\mu_1}(M+Z) - M^* (b_1 U -\correz{\mu_1}),
\end{align*}
from which, using the fact that $b_1 U^* = \correz{\mu_1}$ and $M^* = 1- U^*$, we have
\begin{align*}
        \dfrac{\text{d}}{\text{d}t} V(U,M) &\le (U-U^*) \left( \dfrac{\correz{\mu_1}}{U} -\correz{\mu_1}\right) - M^* (b_1 U -\correz{\mu_1})\\
        & = \correz{\mu_1} - \correz{\mu_1} U - \dfrac{\correz{\mu_1} U^*}{U} + \correz{\mu_1} U^* - b_1 U + \correz{\mu_1} + b_1 U U^* - U^*\\
        & = \correz{\mu_1} \left(1 - U - \dfrac{U^*}{U} + U^* - \dfrac{U}{U^*} + 1 + U - U^* \right)\\
        & = \correz{\mu_1} \left(2 - \dfrac{U^*}{U}  -  - \dfrac{U}{U^*}   \right)\\
        & = - \correz{\mu_1} \dfrac{( U-U^*)^2}{U U^*} \leq 0,
\end{align*}
which concludes the proof.
\end{proof}

\globalMFEb*

\begin{proof}
     To prove the global asymptotic stability of the MFE we use of the following Lyapunov function:
    \begin{equation*}
        V(U,M,Z) \coloneqq U-U^* \ln U + M + Z - Z^* \ln Z,
    \end{equation*}
    where $U^*$ and $Z^*$ are the corresponding equilibrium values, as in Proposition \ref{global_SFE2}. The result follows
from computations which are almost identical to the ones in Proposition \ref{global_SFE2}, using the hypothesis $b_1 < \dfrac{b_2}{1-KI}$; hence, we do not repeat them.
\end{proof}

\globalEEpat*

\begin{proof}
     To prove the global asymptotic stability of the positive equilibrium $\tilde{E}$ we use of the following Lyapunov function:
    \begin{equation*}
        V(U,L) \coloneqq \dfrac{1}{2}(U-U^*)^2 + L - L^* \ln L,
    \end{equation*}
    where $U^*$ and $L^*$ are the corresponding equilibrium values of $\tilde{E}$. Then
     \begin{align*}
        \dfrac{\text{d}}{\text{d}t} V(U,L) &=  (U-U^*) \dfrac{\text{d}}{\text{d}t} U + \dfrac{\text{d}}{\text{d}t} L -\dfrac{L^*}{L} \left( \dfrac{\text{d}}{\text{d}t}L\right) \\ &=
        (U-U^*) (\correz{\mu_1} - b UL - \correz{\mu_1} U) + \dfrac{L-L^*}{L} (b U L - \correz{\mu_1} L)\\&=
        (U-U^*) (\correz{\mu_1} - b UL - \correz{\mu_1} U) + (L-L^*)(b U  - \correz{\mu_1} ),
        \end{align*}
        from which, using the fact that $\correz{\mu_1} = (bL^* + \correz{\mu_1}) U^*$ and $b U^* = \correz{\mu_1}$, we have
        \begin{align*}
        \dfrac{\text{d}}{\text{d}t} V(U,L) &=  (U-U^*) (b L^* U^* + \correz{\mu_1} U^* - b UL - \correz{\mu_1} U) + (L-L^*) (b U - bU^*) \\ 
        &=  b(U-U^*) (L^* U^* - UL) - \correz{\mu_1}(U- U^*)^2 - \correz{\mu_1} (L-L^*) (U^*- U)\\
        & \leq (U-U^*) (b(U^*L^* - UL) + \correz{\mu_1} (L-L^*) ) \\
        & = (U-U^*) (b U^* L^* - b UL + b U^* L - b U^* L^*)\\
        &= - b L (U-U^*)^2 \leq 0,
        \end{align*}
        which concludes the proof.
\end{proof}

{\footnotesize
	\bibliographystyle{unsrt}
	\bibliography{biblio}
}

%

\end{document}